\documentclass[11pt]{article}

\usepackage{amsfonts}
\usepackage{amsmath}
\usepackage{amssymb}
\usepackage{wasysym}
\usepackage{yfonts}  
\usepackage{amsthm}
\usepackage{amscd}   
\usepackage[all]{xypic}
\usepackage{mathrsfs}
\usepackage{indentfirst}
\usepackage{bbold}
\usepackage{amsxtra}
\usepackage{amsbsy}
\usepackage{amstext}
\usepackage{amsopn}
\usepackage{enumerate}
\usepackage{mathtools}
\usepackage{relsize}

\newenvironment{enumerate*}{
\begin{enumerate}[{\rm (i)}]
  \setlength{\itemsep}{3.5pt}
  \setlength{\parskip}{0pt}
}{\end{enumerate}}

\newenvironment{enumerate!}{
\begin{enumerate}[{\rm (I.)}]
  \setlength{\itemsep}{3.5pt}
  \setlength{\parskip}{0pt}
}{\end{enumerate}}

\newenvironment{enumeratenum}{
\begin{enumerate}[{\rm (1)}]
  \setlength{\itemsep}{3.5pt}
  \setlength{\parskip}{0pt}
}{\end{enumerate}}

\CompileMatrices

\renewcommand*{\thefootnote}{\arabic{footnote}}

\numberwithin{equation}{section}

\newtheorem{theorem}{Theorem}[section]
\newtheorem{lemma}[theorem]{Lemma}
\newtheorem{proposition}[theorem]{Proposition}
\newtheorem{corollary}[theorem]{Corollary} 
\newtheorem{definition}[theorem]{Definition}
\newtheorem{example}[theorem]{Example}
\newtheorem*{remark}{Remark}

\newtheorem*{acknowledgements}{Acknowledgements}

\newtheorem*{thm1}{Theorem 1}

\newtheorem*{thm2}{Theorem 2}
\newtheorem*{thm3}{Theorem 3}
\newtheorem*{thm4}{Theorem 4}
\newtheorem*{thm5}{Theorem 5}

\newtheorem*{thm4.1}{Theorem 4.1}
\newtheorem*{thm4.2}{Theorem 4.2}
\newtheorem*{thm4.6}{Theorem 4.6}
\newtheorem*{prop4.7}{Proposition 4.7}
\newtheorem*{prop4.8}{Proposition 4.8}
\newtheorem*{prop4.9}{Proposition 4.9}
\newtheorem*{thm4.10}{Theorem 4.10}

\newtheorem*{prop3.4'}{Proposition 3.4$'$}
\newtheorem*{propA'}{Proposition A$'$}

\usepackage{titlesec}

\titleformat*{\section}{\large \bfseries}
\titleformat*{\subsection}{\bf}
\titleformat*{\subsubsection}{\large\bfseries}
\titleformat*{\paragraph}{\large\bfseries}
\titleformat*{\subparagraph}{\large\bfseries}

\begin{document}

\title{Rank conditions and amenability for rings associated to graphs} 
 
\author{
{\sc Karl Lorensen}\\
{\sc Johan \"Oinert}
}

\maketitle

\begin{abstract} 

We study path rings, Cohn path rings, and Leavitt path rings associated to directed graphs, with coefficients in an arbitrary ring $R$.  For each of these types of rings, we stipulate conditions on the graph that are necessary and sufficient to ensure that the ring satisfies either the rank condition or the strong rank condition whenever $R$ enjoys the same property. In addition, we apply our result for path rings and the strong rank condition to characterize the graphs that give rise to amenable path algebras and exhaustively amenable path algebras.

\vspace{10pt}

\noindent {\bf Mathematics Subject Classification (2020)}:  16P99, 16S88, 43A05, 43A07 
\vspace{5pt}

\noindent {\bf Keywords}:  path ring, path algebra, Cohn path ring, Cohn path algebra, Leavitt path ring, Leavitt path algebra, rank condition, unbounded generating number, strong rank condition, invariant basis number, amenable algebra, exhaustively amenable algebra, algebraically amenable algebra, properly algebraically amenable algebra. 
\end{abstract}
\let\thefootnote\relax

\section{Introduction}

There are three prominent types of rings associated to directed graphs: path rings, Leavitt path rings, and Cohn path rings. Traditionally, such rings are taken to be algebras with coefficients in a field or at least a commutative ring. In the present paper, however, we consider more general coefficient rings and investigate conditions on the graph that guarantee that the rings associated to the graph inherit certain properties from the coefficient ring. The two ring-theoretic properties that we consider are the rank condition and the strong rank condition. Both of these properties have garnered attention recently because of their relevance to the study of amenability (see \cite{Bartholdi2, KL, LO}).  Our results also have  applications to this area, supplying new insights about amenability and exhaustive amenability for path algebras over fields. 

A ring $R$ is said to satisfy the \emph{rank condition} if there does not exist a right $R$-module epimorphism $R^n\to R^{n+1}$ for any positive integer $n$. 
This is equivalent to the property that, for every positive integer $n$, there exists a finitely generated right $R$-module that cannot be generated by $n$ elements. 
For this reason, rings that satisfy the rank condition are often described as having \emph{unbounded generating number}, the term favored by the author of the concept, P. M. Cohn
\cite{CohnIBN, CohnFIR}.

The strong rank condition is the dual of the rank condition; that is, a ring $R$ satisfies the \emph{strong rank condition} if, for any positive integer $n$, there fails to exist a right $R$-module monomorphism $R^{n+1}\to R^n$. An equivalent formulation is that there is no finitely generated right $R$-module that contains an infinite linearly independent set. It is easy to see that, as indicated by the names, the strong rank condition implies the rank condition. But the converse does not hold: every free ring, for instance, satisfies the rank condition but not the strong rank condition.

If $E$ is a directed graph with a finite set $E^0$ of vertices, then we let $RE$ denote the path ring associated to $E$ over the ring $R$. This is the ring consisting of all $R$-linear combinations of paths of finite length, where the product of two paths is formed by concatenation, if that is possible, and defined to be zero otherwise. Our first result, Theorem 1, describes the graphs and coefficient rings that give rise to path rings that fulfill the strong rank condition. In the statement, we refer to a specific quasiordering on $E^0$ as well as to the notions of a hereditary subset of $E^0$ and an \emph{exclusive} cycle, which is one that can only share a vertex with a cycle possessing the same set of edges. 
The quasiordering is defined by $v\geq w$ if there is a path from $v$ to $w$. Moreover, a subset $X$ of $E^0$ is \emph{hereditary} if $v\in X$ whenever there is a vertex $x\in X$ such that $x\geq v$. The quasiordering on vertices induces one on cycles, and cycles that are maximal with respect to this quasiordering play a prominent role in Theorem 1 and throughout the paper. 

\begin{thm1}  Let $R$ be a ring and E a directed graph with finitely many vertices. Let $X$ be the smallest hereditary subset of $E^0$ that contains all the vertices of all the cycles of $E$. Then the statements (i), (ii), and (iii) below are equivalent.
\begin{enumerate*}
\item $RE$ satisfies the strong rank condition. 
\item $R$ satisfies the strong rank condition, and there is no right $RE$-module monomorphism $\left (RE\right )^2\to RE$.
\item $R$ satisfies the strong rank condition, and at least one of the assertions (1) or (2) holds:
\begin{enumeratenum}
\item $X\neq E^0$;
\item $E$ contains an exclusive maximal cycle.
\end{enumeratenum}
\end{enumerate*}
\end{thm1}

With the aid of Theorem 1, we prove the following result characterizing amenable path algebras over fields (see Definition 2.21). 

\begin{thm2} Let $\mathbb K$ be a field and $E$ a directed graph with $E^0$ finite. Let $X$ be the smallest hereditary subset of $E^0$ that contains all the vertices of all the cycles of $E$. Then the statements (i), (ii), (iii), and (iv) below are equivalent.
\begin{enumerate*}
\item $\mathbb KE$ is amenable.
\item $\mathbb KE$ satisfies the strong rank condition.
\item There is no right $\mathbb KE$-module monomorphism $\left (\mathbb KE \right )^2\to \mathbb KE$.
\item $X\neq E^0$, or $E$ contains an exclusive maximal cycle.
\end{enumerate*}
\end{thm2}

In our third theorem, we obtain a characterization of the graphs that give rise to path algebras that are exhaustively amenable (see Definition 2.21), a property that was introduced by L. Bartholdi \cite{Bartholdi} and also studied in \cite{Amenability, Elek}. Like Theorem~2, Theorem~3 was inspired by \cite[Examples 5.14, 5.15]{Amenability}. For the statement of the theorem, we require the following piece of notation: if $E$ is a directed graph and $X\subseteq E^0$, then we write $T(X):=\{v\in E^0: \mbox{there is an $x\in X$ such that $x\geq v$}\}.$

\begin{thm3}  Let $\mathbb K$ be a field and $E$ a directed graph with $E^0$ finite and nonempty. Let $X$ be the smallest hereditary subset of $E^0$ that contains all the vertices of all the cycles of $E$. Then $\mathbb KE$ is exhaustively amenable if and only if at least one of the conditions (i), (ii), (iii), or (iv) below is satisfied.
\begin{enumerate*}
\item $X=\emptyset$. 
\item $T(E^0\setminus X)$ contains a vertex that emits infinitely many edges.
\item $T(E^0\setminus X)$ contains a vertex that belongs to a cycle. 
\item  $E$ contains an exclusive maximal cycle.
\end{enumerate*}
\end{thm3}

Leavitt path rings and Cohn path rings are quotients of path rings; see Section 2 for the specific relations that are divided out in each case. The former rings, introduced in \cite{Abrams1, Ara2} for row-finite graphs and in \cite{Abrams2, Tomforde2} for general graphs, can be interpreted as algebraic versions of the Cuntz-Krieger graph $C^\ast$-algebras described in \cite{Raeburn}.  
Leavitt path rings have as special cases some of the universal rings studied by W. G. Leavitt in \cite{Leavitt2, Leavitt3, Leavitt} that fail to have invariant basis number. (A ring $R$ has \emph{invariant basis number} if $R^n\cong R^m\Leftrightarrow n=m$.)
Cohn path rings were 
first defined in \cite{Ara3} and obey one fewer relation than the Leavitt path rings associated with the same graph. These rings bear Cohn's name since they have as special cases the rings introduced by him in \cite{CohnIBN} that possess invariant basis number but fail to satisfy the rank condition. 

The statements of Theorems 1 and 2 (but not Theorem 3) also hold if the path ring $RE$ and the path algebra $\mathbb KE$ are replaced, respectively, by the Cohn path ring $C_R(E)$ and Cohn path algebra $C_{\mathbb K}(E)$ (see Corollary 4.10 and Theorem 4.12). Furthermore, corresponding results, with the same condition on the graph, hold for Cohn path rings and the rank condition (see Corollary 4.11 and Theorem 4.12). For a path ring $RE$, on the other hand, the rank condition is always inherited from the coefficient ring $R$, regardless of the structure of the graph $E$ (see Proposition 3.1).

Theorem 4 below describes the graphs $E$ that give rise to Leavitt path rings $L_R(E)$ that satisfy the strong rank condition if the coefficient ring $R$ does. These graphs are characterized by conditions closely resembling the two in Theorem 1(iii). The only difference is that the set $X$ must be taken to be larger:  it is now the smallest hereditary and saturated set of vertices containing all the vertices of all the cycles. By \emph{saturated}, we mean that, for any regular vertex $v$, if $r\left (s^{-1}\left (v\right )\right )\subseteq X$, then $v\in X$.  Here $r$ and $s$ represent the range and source functions, respectively, from the set of edges $E^1$ to $E^0$. Moreover, a vertex $v$ is \emph{regular} if $0<|s^{-1}(v)|<\aleph_0$. 

\begin{thm4}  Let $R$ be a ring and E a directed graph with finitely many vertices. Let $X$ be the smallest hereditary and saturated subset of $E^0$ containing all the vertices of all the cycles of $E$. Then the statements (i), (ii), and (iii) below are equivalent. 
\begin{enumerate*}
\item $L_R(E)$ satisfies the strong rank condition.
\item $R$ satisfies the strong rank condition, and there is no right $L_R(E)$-module monomorphism $L_R(E)\oplus L_R(E)\to L_R(E)$. 
\item $R$ satisfies the strong rank condition, and at least one of the statements (1) or (2) holds:
\begin{enumeratenum}
\item $X\neq E^0$;
\item $E$ contains an exclusive maximal cycle.
\end{enumeratenum}
\end{enumerate*}
\end{thm4}

As articulated in Theorem 5 below, the graph-theoretic conditions in Theorem 4(iii) also characterize the graphs that give rise to Leavitt path rings that satisfy the rank condition. In fact, these same conditions first appeared in the work of P.  Ara, K. Li, F. Lled\'o, and J. Wu  \cite[Theorem 5.10]{Amenability}, who showed that they are also necessary and sufficient to ensure that a Leavitt path algebra over a field is amenable. 

\begin{thm5}  Let $R$ be a ring and E a directed graph with finitely many vertices. Let $X$ be the smallest hereditary and saturated subset of $E^0$ containing all the vertices of all the cycles of $E$. Then the statements (i), (ii), and (iii) below are equivalent. 
\begin{enumerate*}
\item $L_R(E)$ satisfies the rank condition.
\item $R$ satisfies the rank condition, and there is no $L_R(E)$-module epimorphism $L_R(E)\to L_R(E)\oplus L_R(E)$. 
\item $R$ satisfies the rank condition, and at least one of the statements (1) or (2) holds:
\begin{enumeratenum}
\item $X\neq E^0$;
\item $E$ contains an exclusive maximal cycle.
\end{enumeratenum}
\end{enumerate*}
\end{thm5}

The special case of Theorem 5 where $R$ is a field and the graph has finitely many edges was proved in \cite{UGN}. The arguments in that paper rely on the description contained in \cite[Theorem 3.5]{Ara2} of the monoid consisting of the isomorphism classes of finitely generated projective modules for a Leavitt path algebra over a field. However, there is no known, useful description of that monoid for Leavitt path rings over general coefficient rings. Moreover, even if there were such a description, it would not be of any apparent benefit in studying the strong rank condition. Hence, in proving both Theorems 4 and 5, we adopt instead the approach employed by Ara, Li, Lled\'o, and Wu \cite{Amenability} to deduce the characterization of amenable Leavitt path algebras alluded to above. We establish a key result, Proposition 2.20, that is specifically designed to allow us to adapt their arguments to our context and aims. 
Two important preliminary results, Proposition 4.3(ii) and Lemma 4.7, required for the proofs of Theorems 4 and 5, are also drawn from \cite{Amenability}. 

\begin{acknowledgements}{\rm The first author thanks Blekinge Tekniska H\"ogskola in Karlskrona, Sweden for generously hosting him from February to May 2023, when some of the groundwork was laid for this paper.

Both authors also wish to underscore the enormous influence that Ara, Li, Lled\'o, and Wu's paper \cite{Amenability} had on the present work. This applies particularly, but not solely, to their results and arguments pertaining to Leavitt path algebras and their two illuminating examples of amenable path algebras. 

Finally, we thank an anonymous referee for their helpful feedback.
}
\end{acknowledgements}

\section{Background and preliminary results}

In this section, we explain our terminology and notation, as well as expound some basic facts about the rank condition, the strong rank condition, and rings associated to graphs. 
Before delving into rings, we define the notion of a quasiordering, also known as a preordering, which plays a salient role in all of our theorems and their proofs. 
A \emph{quasiordering} $\geq$ on a set $X$ is a relation that is reflexive and transitive. A {\it maximal} element of a set $X$ with respect to a quasiordering $\geq$ on $X$ is an element $m\in X$ such that, for all $x\in X$, 
\[x\geq m\ \Longrightarrow\ m\geq x.\]
Obviously, quasiorderings on finite sets always admit at least one maximal element. This fact will be relied upon repeatedly in the sequel. 

The present paper is exclusively concerned with rings with a multiplicative identity element, as the rank and strong rank conditions have historically only been studied for such rings. Indeed, some complications arise in trying to extend these notions to rings without identity. Due to this focus, we will always use the term {\it ring} to mean an associative ring $R$ with a multiplicative identity element, which we also call a {\it unit element} and denote by either $1$ or $1_R$. Furthermore, unless we explicitly state otherwise, a subring must possess the same unit element as the overring, and ring homomorphisms will be required to map the unit element of the domain to that of the codomain. 

By a {\it module}, we will mean a right module unless we indicate otherwise. Also, the unit element of our ring will always be assumed to be an identity element for our module multiplication as well. 

If $R$ is a ring and $n$ a positive integer, then $M_n(R)$ represents the ring of $n\times n$ matrices over $R$.

An {\it $R$-ring} $S$ is a ring that is also an $R$-bimodule such that the following three conditions are satisfied for all $r\in R$ and $s, s'\in S$:
\begin{enumerate*}
\item $(ss')r=s(s'r)$;
\item $(rs)s'=r(ss')$;
\item $(sr)s'=s(rs')$.
\end{enumerate*}
Notice that (iii) implies that $r\cdot 1_S=1_S\cdot r$ for all $r\in R$. If $R$ is commutative and $sr=rs$ for all $r\in R$ and $s\in S$, then $S$ is referred to as an {\it $R$-algebra}. 

Let $R$ be a ring, and let $S, T$ be $R$-rings. The ring $S\times T$ can be made into an $R$-ring by giving it the $R$-bimodule structure of $S\oplus T$. The notation that we use for this $R$-ring is also $S\oplus T$.  

We will now define the various rings that are associated to a directed graph $E$. Since we are only interested in rings with unit elements, it is imperative that we confine ourselves to graphs with finitely many vertices. Denote the sets of vertices and edges of $E$ by $E^0$ and $E^1$, respectively, and assume that $E^0$ is finite. Furthermore, let $r,s:E^1\to E^0$ be the range and source functions, respectively. A {\it path} of {\it length} $n\geq 1$ in $E$ is a sequence of $n$ edges in which the range of each edge except the last coincides with the source of the succeeding edge. If $\alpha$ is such a path, then the {\it source}, $s(\alpha)$, of $\alpha$ is defined to be the source of the first edge in $\alpha$, and the {\it range}, $r(\alpha)$, of $\alpha$ is the range of the last edge in $\alpha$. 

A {\it path of length $0$} is just a vertex; also, for any $v\in E^0$, we define $s(v):=r(v):=v$. For any path $\alpha$, we employ $|\alpha|$ to denote the length of $\alpha$.  In addition, the set of vertices on a path $\alpha$ is denoted ${\rm Vert}(\alpha)$ and the set of edges by ${\rm Edge}(\alpha)$. The range of a path is also referred to as its \emph{terminal vertex}, or \emph{terminus}. 

 A {\it cycle} is a path $\omega:=(e_1,\dots,e_n)$ of length $n>0$ such that $s(e_1)=r(e_n)$ and $s(e_i)=s(e_j)\Longleftrightarrow i=j$. The vertex $s(e_1)$ is called the {\it basepoint} of $\omega$. 
 
Let $v\in E^0$. If $s^{-1}(v)$ is infinite, then $v$ is called an {\it infinite emitter}. If $r^{-1}(v)$ is infinite, then $v$ is an {\it infinite receiver}. If $s^{-1}(v)$ is both nonempty and finite, then $v$ is referred to as a {\it regular} vertex. The set of all regular vertices is denoted ${\rm Reg}(E)$. 

If $R$ is a ring, then the {\it path ring} of $E$ over $R$, denoted $RE$, is defined to be the $R$-ring generated by the set $E^0\cup E^1$ subject only to the following relations:
\begin{enumerate*}
\item $rv=vr$ and $re=er$ for all $v\in E^0$, $e\in E^1$, and $r\in R$;
\item for all $v, w\in E^0$, $v^2=v$ and 
$vw=0$ if $v\neq w$; 
\item $e=s(e)e=er(e)$ for all $e\in E^1$;
\item $\displaystyle{\sum_{v\in E^0} v = 1}$.
\end{enumerate*} 

Condition (iv) in the above definition may strike the reader as redundant; indeed, it is normally not included in the definition of $RE$. We remind the reader, however, that our rings must always possess a unit element. Hence, since $RE$ is a quotient of the free $R$-ring (with unit element!) on the generating set $E^0\cup E^1$, omitting (iv) would result in a different ring, namely, 
one isomorphic to $RE\times R$. Furthermore, we remark that the sum on the left in (iv) should be interpreted as $0$ if $E^0=\emptyset$. Therefore, in this case, the ring $RE$ is the zero ring. 

We identify paths in the graph $E$ with the products of their edges in the ring $RE$. The set of all paths, including those of length $0$, in $RE$ is denoted ${\rm Path}(E)$. It is quite easy to see that $RE$ is freely generated as both a right and left $R$-module by ${\rm Path}(E)$.  

If $\alpha, \sigma\in {\rm Path}(E)$, then $\sigma$ is a {\it subpath} of $\alpha$ if there are paths $\alpha_1, \alpha_2$ such that $\alpha=\alpha_1\sigma\alpha_2$. If, in addition, we have
that $\sigma\neq \alpha$, then $\sigma$ is called a \emph{proper subpath} of $\alpha$. 

We now define the Leavitt path ring and the Cohn path ring of $E$ over $R$. We will subsume both of these under a more general construction. This generalization, introduced in \cite{Ara3}, can also be found in G. Abrams, Ara, and M. Siles Molina's monograph \cite[Definition 1.5.9]{LPA}, albeit under a different name and notation than we employ here. Let $V$ be a subset of ${\rm Reg}(E)$. Furthermore, for every $e\in E^1$, we let $e^\ast$ be a new symbol, called a {\it ghost edge}. Write $\left (E^1\right )^\ast:=\{e^\ast:e\in E^1\}$. The {\it relative Leavitt path ring} of $E$ over $R$ {\it with respect to $V$}, denoted $L^V_R(E)$, is the $R$-ring generated by $E^0\cup E^1\cup \left (E^1\right )^\ast$  and subject to the relations (i)-(iv) above and the relations (1)-(4) below: 
\begin{enumeratenum}
\item $re^\ast=e^\ast r$ for all $r\in R$; 
\item $e^\ast=r(e)e^\ast=e^\ast s(e)$ for all $e\in E^1$;
\item for all $e,f\in E^1$, $e^\ast f=0$ if $e\neq f$ and $e^\ast e=r(e)$;
\item $\sum_{e\in s^{-1}(v)} ee^\ast = v$ for every $v\in V$. 
\end{enumeratenum}
The ring $L^{\rm Reg(E)}_R(E)$ is called the {\it Leavitt path ring} of $E$ over $R$ and written $L_R(E)$. Moreover, the ring $L^\emptyset_R(E)$ is the {\it Cohn path ring} of $E$ over $R$ and denoted $C_R(E)$. 

We call the reader's attention again to our deviation from the standard nomenclature pertaining to these rings. In \cite{LPA} and much of the rest of the literature, our notion of a relative Leavitt path ring with respect to a subset $V$ of ${\rm Reg}(E)$ is known as a {\it relative Cohn path ring} and denoted $C_R^V(E)$. We prefer our designation because relation (4) played a crucial role in W. G. Leavitt's original constructions \cite{Leavitt3, Leavitt2, Leavitt} of rings that fail to possess invariant basis number. Hence it seems fitting that any structure based upon equation (4) should bear his name. On the other hand, in the case of Cohn's \cite{CohnIBN} constructions, it was precisely the absence of relation (4) that made the rings interesting since it endowed them with invariant basis number, even though they failed to satisfy the rank condition. Another apposite term used for a relative Leavitt path ring is \emph{Cohn-Leavitt path ring}, the original name from \cite{Ara3}.  

If $e_1,\dots, e_n\in E^1$ such that $r(e_i)=s(e_{i+1})$ for $i=1,\dots,n-1$ and $\alpha:=e_1\dots e_n$, then we define 
the {\it ghost path} $\alpha^\ast\in L_R^V(E)$ by $\alpha^\ast:=e_n^\ast\dots e^\ast_1$. Also, for $v\in E^0$, we define $v^\ast:=v$. With this notation, we can describe a subset of $C_R(E)$ that generates $C_R(E)$ freely as an $R$-module. 

\begin{proposition}[{\cite[Proposition 1.5.6]{LPA}}] Let $E$ be a directed graph with $E^0$ finite, and let $R$ be a ring. Then $C_R(E)$ is freely generated as a left and right $R$-module by the set
\[\mathlarger{\Sigma_E}:=\{\alpha \beta^\ast\ :\ \alpha, \beta\in {\rm Path}(E)\ \mbox{\rm and}\ r(\alpha)=r(\beta)\}.\] 
\end{proposition}

The hypothesis in the reference for the above proposition is that $R$ is a field, but the proof provided  is valid for any ring. It follows from the proposition that $\Sigma_E$ also generates $L^V_R(E)$ as an $R$-module, but $\Sigma_E$ may not be $R$-linearly independent if $V\neq \emptyset$. For any choice of $V\subseteq {\rm Reg}(E)$, the elements of $\Sigma_E$ can be multiplied in the ring $L^V_R(E)$ as follows (see \cite[Lemma 1.2.12(i)]{LPA}).

\begin{equation}
\left (\alpha_1\beta_1^\ast\right )\left (\alpha_2\beta_2^\ast\right ) = 
\begin{cases}
\alpha_1\alpha_2^\prime\beta_2^\ast & \mbox{if $\alpha_2=\beta_1\alpha_2^\prime$ for some $\alpha_2^\prime\in {\rm Path}(E)$}\\
\alpha_1\left (\beta_1^\prime\right )^\ast \beta_2^\ast & \mbox{if $\beta_1=\alpha_2\beta_1^\prime$ for some $\beta_1^\prime\in {\rm Path}(E)$}\\
0 & \mbox{otherwise.}
\end{cases}
\end{equation}

The generating set $\Sigma_E$  can be trimmed down to form an $R$-basis for $L_R^V(E)$ for any $V\subseteq {\rm Reg}(E)$. Such a basis is described in \cite[Proposition 1.5.11]{LPA}. Although the proposition in the reference is stated for $R$ a field, the argument applies equally well to an arbitrary ring. 

\begin{proposition}[{\cite[Proposition 1.5.11]{LPA}}]  Let $R$ be a ring and $E$ a directed graph with $E^0$ finite.  Let $\Sigma_E$ be as defined in Proposition 2.1. Furthermore, let $V$ be a subset of ${\rm Reg}(E)$, and, for each $v\in V$, choose an edge $e_v\in s^{-1}(v)$.  Define the set $\Lambda\subseteq L^V_R(E)$ by
\[\Lambda:=\{\alpha e_v e_v^\ast \beta^\ast\ :\ v\in V\ \mbox{\rm and}\ \alpha, \beta\in {\rm Path}(E)\ \mbox{\rm such that}\ r(\alpha)=r(\beta)=v\}.\] 
Then $L_R^V(E)$ is free as a left and right $R$-module on the set $\Sigma_E\setminus \Lambda$.
\end{proposition}

\begin{remark}{\rm Both Propositions 2.1 and 2.2 can also be easily deduced from \cite[Corollary~6.2]{diamond}.} \end{remark}

For any directed graph $E$, we denote the opposite graph by $E^{\rm op}$. Moreover, for any ring $R$, we denote the opposite ring by $R^{\rm op}$. In Lemma 2.3, we state two well-known, elementary facts about opposite rings and graphs that will allow us to formulate dual versions of our main results. 

\begin{lemma} Let $E$ be a directed graph with $E^0$ finite, and let $R$ be a ring. Then the following isomorphic relations between rings hold:
\begin{enumerate*}
\item $R^{\rm op}\left [E^{\rm op}\right ]\cong \left (RE\right )^{\rm op}$;
\item $L^V_{R^{\rm op}}\left(E\right )\cong \left (L^V_R\left (E\right )\right )^{\rm op}$ for any $V\subseteq {\rm Reg}(E)$.
\end{enumerate*}
\end{lemma}

\begin{proof} For (i), notice that the map $E\to E^{\rm op}$ defined by reversing the direction of each edge induces an isomorphism $\left (RE\right )^{\rm op}\to R^{\rm op}\left [E^{\rm op}\right ]$. For (ii), the required isomorphism is induced by mapping $e$ to $e^\ast$ and $e^\ast$ to $e$ for every $e\in E^1$ and every vertex to itself. 
\end{proof}

There is a $\mathbb Z$-grading on $L^V_R(E)$ defined by assigning all the vertices degree $0$, each edge in $E^1$ degree $1$, and each ghost edge in $\left (E^1\right )^\ast$ degree $-1$. 
With this grading, we have the following important theorem about $\mathbb Z$-graded multiplicative and additive maps from Leavitt path rings to arbitrary $\mathbb Z$-graded rings. This theorem was proved for Leavitt path algebras over fields in \cite[Theorem 4.8]{Tomforde1} and for Leavitt path algebras over commutative rings in \cite[Theorem 5.3]{Tomforde2}. Moreover, the latter proof carries over readily to relative Leavitt path rings with coefficients in an arbitrary ring. 

\begin{theorem}[{M. Tomforde \cite[Theorem 5.3]{Tomforde2}}] Let $E$ be a directed graph with finitely many vertices, and let $R$ be a ring. Let $V$ be a subset of ${\rm Reg}(E)$. Let $S$ be a $\mathbb Z$-graded ring and $\phi:L^V_R(E)\to S$ a multiplicative and additive $\mathbb Z$-graded map.  If $\phi(rv)\neq 0$ for all $v\in E^0$ and $r\in R\setminus \{0\}$, then $\phi$ is injective. 
\end{theorem}

We now discuss the rank condition and the strong rank condition in detail and make some preliminary observations about how the two conditions pertain to relative Leavitt path rings. To start, we point out that the definition of the rank condition could equally well have been  formulated in terms of the absence of a \emph{left} $R$-module epimorphism $R^n\to R^{n+1}$, as the left and right versions of this condition are equivalent (a consequence of Proposition 2.10((i)$\Leftrightarrow$(ii)) below). 
 The strong rank condition, on the other hand, is not left-right symmetric. That is, there is a \emph{left strong rank condition} in addition to our (right) strong rank condition, and rings that satisfy one condition may not necessarily satisfy the other (see \cite[Remark 1.32]{Lam}). 
 
 For many of our arguments, it will be convenient to refer to the following rephrasings of the definitions of the rank condition and the strong rank condition.
 
 \begin{lemma} A ring $R$ satisfies the rank condition if and only if, for every pair $m, n\in \mathbb Z^+$ with $m>n$, there is no $R$-module epimorphism $R^n\to R^m$.
 \end{lemma}
 
  \begin{lemma} A ring $R$ satisfies the strong rank condition if and only if, for every pair $m, n\in \mathbb Z^+$ with $m<n$, there is no $R$-module monomorphism $R^n\to R^m$.
 \end{lemma}

In Lemma 2.7, we state a crucial, elementary fact about the rank condition.
 
 \begin{lemma}[{\cite[Proposition 1.23]{Lam}}] Let $R$ and $S$ be rings and $\phi:R\to S$ a ring homomorphism. If $S$ satisfies the rank condition, then $R$ satisfies the rank condition. 
 \end{lemma}
 
Direct products of rings, with finitely many factors, behave the same way with respect to both the rank condition and the strong rank condition. 

\begin{lemma} [{\cite[Lemma 2.11]{LO}, \cite[Proposition 1.33]{Lam}}] Let $R$ and $S$ be rings. Then $R\times S$ satisfies the rank condition (respectively, the strong rank condition) if and only if $R$ or $S$ satisfies the rank condition (respectively, the strong rank condition).
\end{lemma}

In studying rings that fail to satisfy the rank condition, we will make use of the conventional quasiordering $\precsim$ on the set of all finite dimensional square matrices over a ring $R$ (see, for example, \cite{Aranda}).  For $x\in M_m(R)$ and $y\in M_n(R)$, we define $x\precsim y$ if there are an $m\times n$ matrix $a$ over $R$ and an $n\times m$ matrix $b$ over $R$ such that $x=ayb$.  
In addition, 
if $x\in M_m(R)$ and $y\in M_n(R)$, then we write 
\[x\oplus y:=\begin{pmatrix} x & 0\\
0 & y\end{pmatrix}\in M_{m+n}(R).\]
Furthermore, if $x_1,\dots, x_r$ are finite dimensional matrices over $R$ with $r>2$, then $x_1\oplus\cdots \oplus x_r:=\left (x_1\oplus \cdots \oplus x_{r-1}\right )\oplus x_r$.

In Lemma 2.9, we gather together some elementary observations concerning the interaction between the operation $\oplus$ and the quasiordering $\precsim$. These are all well known and easy to establish (see, for instance, \cite{Aranda}). 

\begin{lemma} Let $R$ be a ring. Then the following statements hold for any square matrices $x,y,z,w$ over $R$.
\begin{enumerate*}
\item $\displaystyle{x\oplus y\precsim y\oplus x}$.
\item $\displaystyle{x\precsim x\oplus y}$.
\item If $x\precsim y$ and $z\precsim w$, then $\displaystyle{x\oplus z\precsim y\oplus w}$. 
\item If $x$ and $y$ have the same dimensions, then $\displaystyle{x+y\precsim x\oplus y}$.
\item If $x$ and $y$ have the same dimensions and are orthogonal idempotents, then $\displaystyle{x\oplus y\precsim x+y}$.
\item If $x$ and $y$ have the same dimensions and $xy$ is idempotent, then $\displaystyle{xy\precsim yx}$.
\end{enumerate*}
\end{lemma}

Below we state five characterizations of rings that fail to satisfy the rank condition. The first three may be found in \cite[Section 2]{LO}, and the last two are consequences of (ii) and (iii), respectively. 

\begin{proposition} Let $R$ be a ring. Then the following six statements are equivalent. 
\begin{enumerate*}
\item $R$ fails to satisfy the rank condition. 
\item There are integers $m>n>0$, an $m\times n$ matrix $A$ over $R$, and an $n\times m$ matrix $B$ over $R$ such that $AB=I_m$.
\item There is an integer $n>0$ such that, for every integer $m>n$, there exist an $m\times n$ matrix $A$ and an $n\times m$ matrix $B$ such that $AB=I_m$.
\item There is a positive integer $n$ such that every finitely generated $R$-module can be generated by $n$ elements.
\item There are integers $m>n>0$ such that 
\begin{equation} \underbrace{1_R\oplus \dots \oplus 1_R}_{m}\precsim \underbrace{1_R\oplus \dots \oplus 1_R}_n.\end{equation}
\item There is an integer $n>0$ such that, for every integer $m>n$, relation (2.2) holds. 
\end{enumerate*}
\end{proposition}

In \cite{LO}, the authors introduced the notion of the generating number of a ring in order to distinguish the different ways it might fail to satisfy the rank condition. 

\begin{definition}{\rm If a ring $R$ fails to satisfy the rank condition, then we define its \emph{generating number}, denoted ${\rm gn}(R)$, to be the smallest positive integer $n$ such that there is an $R$-module epimorphism $R^n\to R^{n+1}$. Moreover, if $R$ satisfies the rank condition, then ${\rm gn}(R):=\aleph_0$.}
\end{definition}

In Proposition 2.12, we describe five characterizations of finite generating numbers. The first three may be found in \cite[Section 2]{LO}, and the last two follow from the first and second, respectively. 

\begin{proposition} Let $R$ be a ring that fails to satisfy the rank condition. Then the following five statements hold. 
\begin{enumerate*}
\item The generating number of $R$ is the smallest positive integer $n$ such that there are an integer $m>n$, an $m\times n$ matrix $A$ over $R$, and an $n\times m$ matrix $B$ over $R$ such that $AB=I_m$.
\item The generating number of $R$ is the smallest positive integer $n$ such that, for every integer $m>n$, there exist an $m\times n$ matrix $A$ and an $n\times m$ matrix $B$ such that $AB=I_m$.
\item The generating number of $R$ is the smallest positive integer $n$ such that every finitely generated $R$-module can be generated by $n$ elements. 
\item The generating number of $R$ is the smallest positive integer $n$ such that (2.2) holds for some integer $m>n$.
\item The generating number of $R$ is the smallest positive integer $n$ such that (2.2) holds for every integer $m>n$.
\end{enumerate*}
\end{proposition}

\begin{remark}{\rm Rings with generating number one are often referred to in the literature as \emph{properly infinite} (see, for example, \cite{LPA, Amenability, Aranda}).}
\end{remark}

Our first result identifying relative Leavitt path rings that satisfy the rank condition or the strong rank condition is the following. 

\begin{lemma}  Let $R$ be a ring, and let $E$ be a nonempty directed graph with finitely many vertices and edges and without any cycles. Let $V$ be a set of regular vertices of $E$. Then $L^V_R(E)$ satisfies the rank condition (respectively, the strong rank condition) if $R$ satisfies the rank condition (respectively, the strong rank condition).
\end{lemma}

We deduce Lemma 2.13 from the following more general result.  

\begin{proposition} Let $R$ be a ring and $S$ an $R$-ring that is finitely generated and free as an $R$-module. If $R$ satisfies the rank condition (respectively, the strong rank condition), then $S$ satisfies the rank condition (respectively, the strong rank condition).
\end{proposition}

\begin{proof} 
We have $S\cong R^k$ as $R$-modules for some positive integer $k$. 
Suppose that 
$S$ fails to satisfy the rank condition. According to Lemma 2.5, this means that we can find positive integers $m>n$ such that there is an $S$-module epimorphism $\phi: S^n\to S^m$. Since $\phi$ is also an $R$-module epimorphism, it induces an  $R$-module epimorphism $R^{kn}\to R^{km}$. Therefore $R$ fails to satisfy the rank condition. An analogous argument establishes the statement for the strong rank condition. 
\end{proof}

\begin{proof}[Proof of Lemma 2.13] The hypotheses guarantee that the number of paths in $E$ must be finite. Hence Proposition 2.2 implies that $L:=L^V_R(E)$ is a finitely generated free $R$-module. Thus the conclusion follows immediately from Proposition 2.14.
\end{proof}

Another important observation for our purposes is that, if a path ring or relative Leavitt path ring satisfies the strong rank condition, then its coefficient ring does as well. These two results are special cases of Proposition 2.15 below, which follows immediately from \cite[\S 1, Exercise 20]{Lam}.

\begin{proposition} Let $R$ be a ring and $S$ an $R$-ring that is free as a left $R$-module. If $S$ satisfies the strong rank condition, then $R$ satisfies the strong rank condition.
\end{proposition}

\begin{corollary} Let $R$ be a ring, and let $E$ be a directed graph with $E^0$ finite. 
If $RE$ satisfies the strong rank condition, then $R$ satisfies the strong rank condition.
\end{corollary}

\begin{corollary} Let $R$ be a ring and $E$ a directed graph with $E^0$ finite. Let $V$ be a subset of ${\rm Reg}(E)$. 
 If $L^V_R(E)$ satisfies the strong rank condition, then $R$ must satisfy the strong rank condition.
\end{corollary} 

We point out, for future reference, that the statements of Proposition 2.15, Corollary 2.16, and Corollary 2.17 are also true if the strong rank condition is replaced by the rank condition. This is an immediate consequence of Lemma 2.7. 

An important property shared by both the rank condition and the strong rank condition is Morita invariance. Proving this for the rank condition is an exercise in \cite[Section 0.1]{CohnFIR}; two different solutions to the problem appear in \cite[Theorem 2.8]{UGN} and \cite[Corollary 2.3]{LO}, respectively.  For the strong rank condition, the invariance follows immediately from Lemma 2.19 below, together with the fact that category equivalences preserve monomorphisms and progenerators. In the statement of the lemma, we use $\mathfrak{M}_R$ to denote the category of $R$-modules. 

\begin{proposition} Both the rank condition and the strong rank condition are Morita invariant.
\end{proposition} 

\begin{lemma} A ring $R$ satisfies the strong rank condition if and only if, for any $m,n\in \mathbb Z^+$ and progenerator $P$ in $\mathfrak{M}_R$, there can only be an $R$-module monomorphism $P^n\to P^m$ if $n\leq m$.
\end{lemma}

\begin{proof} The ``if" direction is obvious. We establish the ``only if" part by proving the contrapositive. Hence we begin by supposing that, for some progenerator $P$, there is a monomorphism $P^n\to P^m$ for some pair of positive integers $m,n$ with $n>m$. It follows, then, that there is a monomorphism $P^n\to P^m$ for every $n>m$. 
Let $k,l\in \mathbb Z^+$ such that $P$ is a direct summand in $R^k$ and $R$ is a direct summand in $P^l$. For every $n\in \mathbb Z^+$, we have a chain of monomorphisms
\[R^n\longrightarrow P^{nl}\longrightarrow  P^{ml}\longrightarrow R^{klm}.\]
Therefore $R$ fails to satisfy the strong rank condition, completing the proof. 
\end{proof}

We now establish our most important preliminary result, a crucial tool for the proofs of our theorems about the rank condition and the strong rank condition.  

\begin{proposition} Let $R$ be a ring and $S$ an $R$-ring. Assume that $S$ is freely generated as an $R$-module by a subset $\Sigma$ such that $R\sigma=\sigma R$ for all $\sigma\in \Sigma$.  Suppose further that, for any finite subset
$K$ of $\Sigma$ and any real number $p>1$, there is a finite subset $F$ of $\Sigma$ such that $KF\setminus \{0\}$ is contained in $\Sigma$  and  $|KF\setminus \{0\}|<p|F|$. 
Then the following two assertions are true. 
\begin{enumerate*}
\item If $R$ satisfies the strong rank condition, then so does $S$. 
\item If $R$ satisfies the rank condition, then so does $S$.  
\end{enumerate*}
\end{proposition}

\begin{proof} 
For the proofs of the two statements, we employ the following notation: for any set $X\subseteq \Sigma$ and $r\in \mathbb Z^+$, let $S^r_X$ denote the $R$-submodule of $S^r$ consisting of all the elements whose components reside in the $R$-submodule of $S$ that is generated by $X$. In addition, we write $\pi_X^r:S^r\to S^r_X$ for the projection $R$-module epimorphism.  

We establish statement (i) by proving the contrapositive. 
Suppose that $S$ fails to satisfy the strong rank condition. Then there are positive integers $m<n$ and an $S$-module monomorphism $\phi:S^n\to S^m$. 
For each $j=1,\dots,n$, define ${\bf e}_j$ to be the element of $S^n$ that has $1$ in the $j$th position and $0$ everywhere else. We can find a finite subset $K$ of $\Sigma$ such that $\phi({\bf e}_j)\in S_K^m$ for $j=1,\dots, n$.  

According to the hypothesis, there is a finite subset $F$ of $\Sigma$ such that $P:=KF\setminus \{0\}$ is contained in $\Sigma$ and $m\, |P|< n\, |F|$. 
For any $f\in F$, we have $\phi({\bf e}_jf)=\phi({\bf e}_j)f\in S_P^m$ for  $j=1,\dots, n$.  Since $S_F^n$ is generated as an $R$-module by the subset $\{{\bf e}_jf:f\in F, j=1,\dots, n\}$, it follows that $\phi(S_F^n)\subseteq S_P^m$. But this then means that there is an $R$-module monomorphism $R^{n|F|}\to R^{m|P|}$, implying that $R$ fails to satisfy the strong rank condition. This concludes the proof of (i). 

To establish (ii), we again prove the contrapositive. 
Suppose that $S$ fails to satisfy the rank condition. This means that there are integers $m>n>0$ and an $S$-module epimorphism $\phi:S^n\to S^m$. For each $i=1,\dots,m$, define ${\bf e}_i$ to be the element of $S^m$ that has $1$ in the $i$th position and $0$ everywhere else. Take ${\bf a}_1,\dots, {\bf a}_m\in S^n$ such that $\phi({\bf a}_i)={\bf e}_i$ for every $i=1,\dots, m$.  There is a finite subset $K$ of $\Sigma$ such that, for every $i=1,\dots, m$, we have ${\bf a}_i\in S^n_K$.  Moreover, according to the hypothesis, there is a finite subset $F$ of $\Sigma$ such that $P:=KF\setminus \{0\}$ is contained in $\Sigma$ and
$n\, |P| < m\, |F|.$ 

We claim that $S^m_F\subseteq \phi(S^n_P)$. To see this, let $f\in F$ and notice that 
\[{\bf e}_if=\phi({\bf a}_i)f=\phi({\bf a}_if)\in \phi(S^n_{P})\]
for $i=1,\dots,m$. Since the set $\{{\bf e}_if:f\in F, i=1,\dots, m\}$  generates $S^m_F$ as an $R$-module, this verifies our claim. 

Now let $\phi_{P}:S^n_{P}\to S^m$ be the restriction of $\phi$ to $S^n_{P}$. In view of the above claim, the composition $\pi^m_F\phi_{P}:S^n_{P}\to S^m_{F}$ must be an $R$-module epimorphism. But this then gives rise to an $R$-module epimorphism $R^{n|P|}\to R^{m|F|}$, which means that $R$ fails to satisfy the rank condition. 
\end{proof}

We conclude this section by discussing the properties of amenability and exhaustive amenability for algebras over fields. 
We will adopt the following definitions of these two notions, which are  equivalent to the ones given in \cite{Amenability} and \cite{Bartholdi} for algebras possessing a unit element. 

\begin{definition}{\rm Let $\mathbb K$ be a field and $A$ a $\mathbb K$-algebra. We say that $A$ is {\it amenable} if, for any finite dimensional subspace $U$ of $A$ and real number $p>1$, there is a finite dimensional subspace $V$ of $A$ such that 
\begin{equation} \dim_{\mathbb K} \left (UV\right ) < p\, \dim_{\mathbb K} V.\end{equation} 

We refer to $A$ as {\it exhaustively amenable} if, for every pair of finite dimensional subspaces $U, W$of $A$ and every real number $p>1$, there is a finite dimensional subspace $V$ of $A$ such that $W\subseteq V$ and (2.3) holds. 

If $A^{\rm op}$ is amenable (respectively, exhaustively amenable), then we refer to $A$ as \emph{right amenable} (respectively, \emph{exhaustively right amenable}).
 } 
\end{definition}

Amenability for algebras was introduced by G. Elek \cite{Elek}, who used the term ``amenable" for our notion of exhaustively amenable. His interest, though, was focused on domains, for which the two varieties of amenability are equivalent (see \cite[Corollary 3.10]{Amenability}).  
It was Bartholdi \cite{Bartholdi} who was the first to make the distinction between amenable and exhaustively amenable algebras. The most comprehensive investigation of these two types of algebras was undertaken in \cite{Amenability}, where amenable algebras were referred to as \emph{algebraically amenable} and exhaustively amenable ones as \emph{properly algebraically amenable}.  

In \cite{Amenability}, the following convenient alternative characterization of exhaustive amenability for infinite dimensional algebras was established. We include the proof for the sake of completeness.

\begin{proposition}[{Ara, Li, Lled\'o, Wu \cite[Proposition 3.5(2)]{Amenability}}] Let $\mathbb K$ be a field and $A$ an infinite dimensional $\mathbb K$-algebra. Then $A$ is exhaustively amenable if and only if, for every finite dimensional subspace $U$ of $A$, every integer $N$, and every real number $p>1$, there is a finite dimensional subspace $V$ of $A$ such that $\dim_{\mathbb K} V>N$ and 
 \[\dim_{\mathbb K} \left (UV\right ) < p\, \dim_{\mathbb K} V.\]
 \end{proposition}

\begin{proof} The ``only if" part is trivial. To prove the ``if" assertion, let $U$ and $W$ be finite dimensional subspaces of $A$, and let $p$ be a real number such that $p>1$. Then there is a finite dimensional subspace $V$ of $A$ such that 
 \[\dim_{\mathbb K} V>\frac{2\, \dim_{\mathbb K}\left (UW\right )}{p-1}\ \ \mbox{and}\ \ \dim_{\mathbb K} \left (UV\right ) < \frac{p+1}{2}\, \dim_{\mathbb K} V.\]
 Set $V^\prime:=V+W$. Then
 \[\dim_{\mathbb K} \left (UV^\prime\right ) \leq \dim_{\mathbb K}\left (UV\right )+\dim_{\mathbb K}\left (UW\right )<\frac{p+1}{2}\, \dim_{\mathbb K} V+\frac{p-1}{2}\, \dim_{\mathbb K} V\leq p\, \dim_{\mathbb K} V^\prime.\]
Therefore $A$ is exhaustively amenable.
 \end{proof}

In the next section, we will find the following proposition useful. 

\begin{proposition} Let $\mathbb K$ be a field and $p$ a real number larger than $1$. Then the following two statements hold. 
\begin{enumerate*}
\item A $\mathbb K$-algebra $A$ is not amenable if and only if there is a finite dimensional subspace $U$ of $A$ such that
\[\dim_{\mathbb K} \left (UV\right )\geq p\, \dim_{\mathbb K} V\]
for every finite dimensional subspace $V$ of $A$.
\item An infinite dimensional $\mathbb K$-algebra $A$ is not exhaustively amenable if and only if there is a finite dimensional subspace $U$ of $A$ and a positive integer $N$ such that
\[\dim_{\mathbb K} \left (UV\right )\geq p\, \dim_{\mathbb K} V\]
for every finite dimensional subspace $V$ of $A$ for which $\dim_{\mathbb K} V > N$.
\end{enumerate*}
\end{proposition}

\begin{proof} We confine ourselves to proving (i); our argument can easily be modified, with the aid of Proposition 2.22, to yield a proof of (ii). The ``if" statement follows immediately from the definition of amenability. To prove the ``only if" part, we assume that $A$ is not amenable. This means that, for some real number $q>1$, there is a finite dimensional subspace $W$ of $A$ such that $\dim_{\mathbb K} \left (WV\right )\geq q\, \dim_{\mathbb K} V$ for every finite dimensional subspace $V$ of $A$. It follows that $\dim_{\mathbb K} \left (W^nV\right )\geq q^n\, \dim_{\mathbb K} V$ for every positive integer $n$ and finite dimensional subspace $V$. Thus, if we choose $n$ so that $q^n\geq p$ and set $U:=W^n$, then we have $\dim_{\mathbb K} \left (UV\right )\geq p\, \dim_{\mathbb K} V$ for every finite dimensional subspace $V$.
\end{proof}

We now connect the property of amenability to the strong rank condition, proving that every amenable algebra over a field satisfies the strong rank condition. This result generalizes \cite[Corollary 4.9]{Amenability}, where it is shown that the generating number of an amenable algebra is greater than one. 

\begin{proposition} Let $\mathbb K$ be a field and $A$ a $\mathbb K$-algebra. If $A$ is amenable, then $A$ must satisfy the strong rank condition.
\end{proposition}

\begin{proof}
Suppose that $A$ fails to satisfy the strong rank condition. Then there are positive integers $m<n$ and an $A$-module monomorphism $\phi:A^n\to A^m$. 
For each $j=1,\dots,n$, define ${\bf e}_j$ to be the element of $A^n$ that has $1$ in the $j$th position and $0$ everywhere else. Also, for $i=1,\dots, m$, let $\pi_i:A^m\to A$ be the $A$-module homomorphism defined by projecting onto the $i$th summand. 

Let $U$ be the $\mathbb K$-subspace of $A$ spanned by the elements $\pi_i\phi({\bf e}_j)$ for $i=1,\dots, m$ and $j=1,\dots, n$. 
Since $A$ is amenable, there is a finite dimensional $\mathbb K$-subspace $V$ of $A$ such that 
\[ m\, \dim_{\mathbb K} (UV) < n\, \dim_{\mathbb K}(V).\]

For any $v\in V$, we have $\pi_i\phi({\bf e}_jv)=\pi_i\phi({\bf e}_j)v\in UV$ for $i=1,\dots, m$ and $j=1,\dots, n$.  As a result, $\phi$ induces a $\mathbb K$-linear monomorphism $V^n\to (UV)^m$. But this means 
\[ m\, \dim_{\mathbb K} (UV) \geq n\, \dim_{\mathbb K}(V),\]
a contradiction. Therefore $A$ satisfies the strong rank condition.
\end{proof}

For the proofs of Theorems 2 and 3, we will require the following easy lemma that can be employed to prove that an algebra is amenable or exhaustively amenable.

\begin{lemma} Let $\mathbb K$ be a field and $A$ a $\mathbb K$-algebra with $\mathbb K$-basis $\Sigma$. Suppose that, for any real number $p>1$ and finite subset $K$ of $\Sigma$, there is a finite subset $F$ of $\Sigma$ such that $|KF\setminus\{0\}|<p|F|$. Then $A$ is amenable. 

If, in addition, the subset $F$ can be chosen so that it is arbitrarily large, then $A$ is exhaustively amenable. 
\end{lemma}

\begin{proof} Let $p$ be a real number greater than $1$, and let $U$ be a finite dimensional subspace of $A$. Then we can find a finite subset $K$ of $\Sigma$ such that $U\subseteq {\rm span}_{\mathbb K} K$. Moreover, the hypothesis furnishes a finite subset $F$ of $\Sigma$ such that $|KF\setminus\{0\}|<p|F|$. Let $V$ be the subspace of $A$ spanned by $F$. Then 
\[\dim_{\mathbb K} (UV) \leq |KF\setminus\{0\}|<p|F|=p\, \dim_{\mathbb K} V.\]
Therefore $A$ is amenable. Moreover, if $F$ can be chosen to be arbitrarily large, then $\dim_{\mathbb K} V$ can be made arbitrarily large, in which case, $A$ will be exhaustively amenable. 
\end{proof}

\section{Path rings}

In this section, we establish our results about path rings: Theorem 1 is about the strong rank condition, and Theorems 2 and 3 deal with amenability and exhaustive amenability, respectively.  
We begin by observing that every path ring satisfies the rank condition if its coefficient ring does.

\begin{proposition} Let $E$ be a nonempty directed graph with $E^0$ finite, and let $R$ be a ring. Then $RE$ satisfies the rank condition if and only if $R$ satisfies the rank condition.
\end{proposition}
\begin{proof} The ``only if" assertion is an immediate consequence of Lemma 2.7. To establish the ``if" statement, assume that $R$ satisfies the rank condition. Let $v\in E^0$ and define the $R$-module epimorphism $\pi_v:RE\to R$ by, for every $x\in RE$, taking $\pi_v(x)$ to be the coefficient of $v$ in the unique representation of $x$ as a linear combination of the elements of ${\rm Path}(E)$. Then $\pi_v:RE\to R$ is also a ring homomorphism.
Hence, by Lemma 2.7, $RE$ fulfills the rank condition.
\end{proof}

For the proof of Theorem 1, we will employ the following notation: if $X$ is a set of vertices of a directed graph $E$, then we write $M(X):=\{v\in E^0:v\geq x\, \ \mbox{for some $x\in X$}\}$. 

\begin{proof}[Proof of Theorem 1]  The implication (i)$\implies$(ii) is a consequence of Corollary 2.16 and the definition of the strong rank condition. We establish (ii)$\implies$(iii) by proving its contrapositive. Hence assume that $R$ fails to satisfy the strong rank condition or both (1) and (2) are false. In the former case, statement (ii) is plainly false. Hence we must just treat the case where $X=E^0\neq \emptyset$ and every maximal cycle of $E$ is nonexclusive. There are, then, distinct maximal cycles $\omega_1,\dots,\omega_n, \xi_1,\dots,\xi_n$ with the following properties.
\begin{itemize}
\item For $i=1,\dots, n$, the cycles $\omega_i$ and $\xi_i$ are based at the same vertex $v_i$.
\item For every cycle $\theta$, there is an integer $i\in [1,n]$ such that $\omega_i\geq \theta$.
\end{itemize}
 
 For each $i=1,\dots,n$, let $v_{i1},v_{i2},\dots,v_{im_i}$ be a complete list of the distinct vertices of $E$ that can be reached from $v_i$ along a path in the graph. Then $E^0=\{v_{ij}:{1\leq i\leq n,}\, {1\leq j\leq m_i}\}$. For each $i=1,\dots, n$ and $j=1,\dots, m_i$, select a single path $\mu_{ij}$ from $v_i$ to $v_{ij}$.
Also, for every $i=1,\dots, n$, set
\[a_i:=\sum_{j=1}^{m_i}\omega_i^j\, \xi_i\, \mu_{ij},\ \ \ \ b_i:=\sum_{j=1}^{m_i}\xi_i^j\, \omega_i\, \mu_{ij}.\]
We claim that the equation 
\begin{equation}
\left (\sum_{i=1}^n a_i\right )x+ \left (\sum_{i=1}^n b_i\right )y = 0
\end{equation}
has no nonzero solution $(x,y)$ in $RE$, implying that there is an $RE$-module monomorphism $\left (RE\right )^2\to RE$. To establish the claim, suppose that $(x,y)$ satisfies (3.1) and write 
\[x:=\sum_{\gamma\in {\rm Path}(E)} \gamma\, r_{\gamma},\ \ \ \ \ y:=\sum_{\gamma\in {\rm Path}(E)} \gamma\, s_{\gamma},\]
where $r_{\gamma}, s_{\gamma}\in R$ for all $\gamma\in {\rm Path}(E)$, with at most finitely many of these elements of $R$ nonzero. 

For every $\gamma\in {\rm Path}(E)$
and $i=1,\dots,n$, we define
\[\mathcal{A}^{\gamma}_i:= \{\omega_i^j\, \xi_i\, \mu_{ij}\, \gamma:j=1,\dots, m_i\}\setminus \{0\},\ \ \ \ \mathcal{B}^{\gamma}_i:= \{\xi_i^j\, \omega_i\, \mu_{ij}\, \gamma:j=1,\dots,m_i\}\setminus \{0\}.\]
Observe that, for each $\gamma\in {\rm Path}(E)$ and $i=1,\dots,n$, the sets $\mathcal{A}^{\gamma}_i$ and $\mathcal{B}^{\gamma}_i$ are either both empty or both singletons. Notice further that, if $\gamma\in {\rm Path}(E)$, then there exists at least one integer $i\in [1,n]$ such that $\mathcal{A}_i^{\gamma}$ and $\mathcal{B}_i^{\gamma}$ are nonempty. 
In addition, the following three assertions hold for all $\gamma, \delta\in {\rm Path}(E)$ and $i,k=1,\dots,n$:
\begin{itemize}
\item $\mathcal{A}^\gamma_i\cap \mathcal{B}^\delta_k = \emptyset$;
\item $\mathcal{A}^\gamma_i\cap \mathcal{A}_k^\delta = \emptyset$ if $(i,\gamma)\neq (k,\delta)$; 
\item $\mathcal{B}^\gamma_i\cap \mathcal{B}_k^\delta = \emptyset$ if $(i,\gamma)\neq (k,\delta)$. 
\end{itemize}

For every $\gamma\in {\rm Path}(E)$ and $i=1,\dots, n$, let $\alpha_i^\gamma$ and $\beta_i^\gamma$ be the unique elements of $\mathcal{A}_i^\gamma$ and $\mathcal{B}_i^\gamma$, respectively, if these two sets are nonempty and set $\alpha_i^\gamma:=\beta_i^\gamma:=0$ if $\mathcal{A}_i^\gamma=\mathcal{B}_i^\gamma=\emptyset$ . 
With this notation, equation (3.1) can be rearranged to form the equation
\[\mathlarger{\mathlarger{\sum}}_{\gamma\in {\rm Path}(E)}\sum_{i=1}^n \left (\alpha_i^\gamma r_{\gamma} +  \beta_i^\gamma s_{\gamma}\right ) = 0.
\]
As a result, we have $r_{\gamma}=s_{\gamma}=0$ for all $\gamma\in {\rm Path}(E)$; that is, $x=y=0$.  This completes the proof of (ii)$\implies$(iii). 

It remains to show that (iii)$\implies$(i). Assume that $R$ satisfies the strong rank condition and that (1) or (2) is true. First we treat the case where (1) holds and $E^0\setminus X$ contains an infinite emitter.  We can thus find a maximal infinite emitter $v$ in $E^0\setminus X$ with respect to the quasiordering.  Let $F$ be the set of all paths that end at $v$. Then the paths in $F$ contain no cycles and no infinite emitters other than $v$. As a result, $F$ must be finite. Also, $KF\setminus \{0\}\subseteq F$ for every finite set $K\subseteq {\rm Path}(E)$. 
It follows, then, from Proposition 2.20 that $RE$ fulfills the strong rank condition. 

Next we assume that (1) is true and $E^0\setminus X$ has no infinite emitters. In this case, we let $F$ be the set of all paths whose vertices lie outside of $X$. Because none of the paths in $F$ contain either cycles or infinite emitters, the number of paths in $F$ must be finite. Moreover, if $K$ is an arbitrary finite subset of ${\rm Path}(E)$, then $KF\setminus \{0\}\subseteq F$. Therefore, by Proposition 2.20, $RE$ satisfies the strong rank condition.  

Finally, we handle the case that (2) holds and (1) does not. Let $\omega$ be an exclusive maximal cycle, and let $v$ be the basepoint of $\omega$. We maintain that $M(v)={\rm Vert}(\omega)$.
To show this, let $w\in M(v)$. Since $X=E^0$, there is a cycle $\xi$ and a path $\delta$ that starts on $\xi$ and ends at $w$. But then the maximality of $\omega$ implies the existence 
of a path from $v$ to $\xi$ and thus one from $v$ to $w$. Since $\omega$ is exclusive, we conclude that $w\in {\rm Vert}(\omega)$. Hence $M(v)={\rm Vert}(\omega)$, as claimed. The exclusivity of $\omega$ also guarantees that there are no edges connecting vertices of $\omega$ that are not part of $\omega$. As a result, the only paths terminating in $v$ that fail to contain $\omega$ are proper subpaths of $\omega$. We will denote this finite set of paths by~$P$. 

Our aim is to again employ Proposition 2.20 to establish that $RE$ fulfills the strong rank condition. To this end, we let $K$ be an arbitrary finite subset of ${\rm Path}(E)$, and let $p$ be a real number greater than $1$. If at least one of the paths in $K$ has $\omega$ as a subpath, then define $k$ to be the largest positive integer such that $\omega^k$ is a subpath of some path in $K$.  Moreover, if $\omega$ is not a subpath of any of the paths in $K$, then we put $k:=0$. 
Next choose $l$ to be a positive integer such that $\frac{k+l+1}{l}<p$, and define $F$ to be the set of all paths of the form $\gamma \omega^i$, where $1\leq i\leq l$ and  $\gamma\in P$. 

We will show that 
\begin{equation} KF\setminus \{0\}\subseteq \{\gamma \omega^i\ :\ 1\leq i\leq k+l+1\ \ \mbox{and}\ \ \gamma\in P\}.\end{equation}
To establish this containment, let $\alpha\in K$, $\gamma\in P$, and $i\in \mathbb Z^+$ such that $i\leq l$ and $\alpha\gamma\omega^i\neq 0$.  If $\alpha$ fails to contain $\omega$ as a subpath, then $\alpha\gamma\omega^i$  plainly resides in the set on the right of the containment expressed in (3.2). Otherwise there are a positive integer $n\leq k$ and paths $\lambda, \mu$  that fail to contain $\omega$ and such that $\alpha=\lambda\omega^n\mu$. Moreover, we have $s(\mu\gamma)=r(\mu\gamma)=v$. It follows, then, that either $\mu=\gamma=v$ or $\mu\gamma=\omega$. As a result, we have that either $\alpha\gamma\omega^i=\lambda\omega^{n+i}$ or $\alpha\gamma\omega^i=\lambda\omega^{n+i+1}$.
Therefore (3.2) holds. 

Statement (3.2) implies that $|KF\setminus \{0\} |\leq |P|(k+l+1)$. Since $|F|=|P|\, l$, we conclude that $|KF\setminus \{0\}|\leq \frac{k+l+1}{l}|F|<p|F|$. An appeal to Proposition 2.20, then, yields that $RE$ satisfies the strong rank condition. 
\end{proof}

 Applying Lemma 2.3(i), we can enunciate a version of Theorem 1 that refers to the left strong rank condition instead. 

\begin{corollary} Let $R$ be a ring and E a directed graph with finitely many vertices. Let $X$ be the smallest subset of $E^0$ such that $M(X)=X$ and $X$ contains all the vertices of all the cycles of $E$. Then the statements (i), (ii), and (iii) below are equivalent.
\begin{enumerate*}
\item $RE$ satisfies the left strong rank condition. 
\item $R$ satisfies the left strong rank condition, and there is no left $RE$-module monomorphism $\left (RE\right )^2\to RE$.
\item $R$ satisfies the left strong rank condition, and at least one of the assertions (1) or (2) holds:
\begin{enumeratenum}
\item $X\neq E^0$;
\item $E$ contains an exclusive minimal cycle.
\end{enumeratenum}
\end{enumerate*}
\end{corollary}

We now apply Theorem 1 to prove Theorems 2 and 3 from the introduction, which characterize amenable and exhaustively amenable path algebras, respectively. 

\begin{proof}[Proof of Theorem 2] The implication (i)$\implies$(ii) is a consequence of Proposition 2.24;
also, (ii)$\implies$(iii) is trivial. Furthermore, Theorem 1((ii)$\implies$(iii)) implies (iii)$\implies$(iv). Finally, the reasoning employed to prove the (iii)$\implies$(i) portion of Theorem 1 can be used to show (iv)$\implies$(i). For this, we simply need to substitute an appeal to Lemma~2.25 for each invocation of Proposition 2.20. 
\end{proof}

For the proof of Theorem 3, we require the notion of the restriction of a graph $E$ to a subset of $E^0$. 

\begin{definition}{\rm  Let $E$ be a directed graph and $X$ a subset of $E^0$. The \emph{restriction} of $E$ to $X$, denoted $E_X$, is the graph defined by 

\[\left (E_X\right )^0:=X; \ \ \ \left (E_X\right )^1:=\{e\in E^1\ :\ s(e)\in X\  \mbox{\rm and}\ r(e)\in X\},\]
where the range and source functions of $E_X$ are the restrictions of the ones for $E$.
}
\end{definition}

\begin{proof}[Proof of Theorem 3] We begin by establishing the ``if " part of the equivalence. Assume first that (ii) or (iii) holds. We pick a vertex $v$ in $T(E^0\setminus X)$ that is a maximal element of this set with respect to the property of lying on a cycle or being an infinite emitter.  Note that this means that any infinite emitter in $\left (M(v)\cap T(E^0\setminus X)\right )\setminus \{ v\}$ must necessarily lie on a cycle based at $v$. In addition, let us make our choice of the vertex $v$ so that, if $v\in X$, there is at least one path from $E^0\setminus X$ to $v$ such that every nonterminal vertex on that path fails to belong to a cycle. The latter condition can be ensured by, if necessary, replacing $v$ by the first vertex along a path from $E^0\setminus X$ to $v$ that is located on a cycle.

 Now let $P$ be the set of all paths $\gamma$ from $E^0\setminus X$ to $v$ such that $\gamma$ fails to possess any nonterminal vertex that lies on a cycle. It follows that every path in $P$ also fails to have any nonterminal vertices that are infinite emitters. The absence of cycles and nonterminal vertices that are infinite emitters along the paths in $P$  guarantees that the set $P$ must be finite. 
Another property that $P$ enjoys is that 
\begin{equation} \left ({\rm Path}(E)\right )P\subseteq P\cup \{0\}.\end{equation} 
To see this, suppose that $\alpha\gamma\neq 0$ for some path $\alpha$ and $\gamma\in P$. Then $r(\alpha)=s(\gamma)\notin X$, which means that all the vertices on $\alpha$ are situated outside of $X$. As a result, there are no vertices along $\alpha$ that belong to cycles. Therefore $\alpha\gamma\in P$, as desired.

Bearing in mind that $0<|P|<\aleph_0$ and (3.3) holds, we proceed to show that $\mathbb KE$ is exhaustively amenable. 
First we treat the case where $v$ is an infinite emitter. Take $K$ to be an arbitrary finite set of paths in $E$, and let $N$ be a positive integer. Choose distinct edges $e_1,\dots, e_n$ such that $n>N$ and $s(e_i)=v$ for $i=1,\dots,n$.  Let $F$ be the set of all paths of the form $\gamma e_i$ for $\gamma\in P$ and $i=1,\dots, n$. Then $N<|F|<\aleph_0$, and, by (3.3), $KF\setminus \{ 0\} \subseteq F$. Therefore, according to Lemma~2.25, the algebra $\mathbb KE$ is exhaustively amenable. 

A similar argument disposes of the case where $v$ lies on a cycle $\omega$. Under this assumption, let $K$ be again an arbitrary finite subset of ${\rm Path}(E)$ and $N$ an arbitrary positive integer. This time, define $F$ to be the set of all paths of the form $\gamma\omega^i$, where $0\leq i\leq N$ and $\gamma\in P$. As before, it follows from (3.3) that $KF\setminus \{0\}\subseteq F$. Since we also have $N<|F|<\aleph_0$, an appeal to Lemma~2.25 yields that $\mathbb KE$ is exhaustively amenable. 

Suppose next that (i) is true and (ii) is false. Then $E$ has no infinite emitters and no cycles. As a result, ${\rm Path}(E)$ is finite, so that $\mathbb KE$ is finite dimensional over $\mathbb K$.
Thus $\mathbb KE$ is exhaustively amenable. 

Assume now that statement (iv) holds and (iii) does not. 
Let $\omega$ be an exclusive maximal cycle with basepoint $v$. Since (iii) is false, we must have $M(v)\subseteq X$. Thus the same reasoning that was employed in the fourth-to-last paragraph of the proof of Theorem 1 establishes that all the paths with terminus $v$ that fail to contain $\omega$ are proper subpaths of $\omega$. 
 Furthermore, we can use exactly the same arguments as in the final three paragraphs of the proof of Theorem 1 to show that, for any finite set $K\subseteq {\rm Path}(E)$ and real number $p>1$, there is a finite set $F\subseteq {\rm Path}(E)$ such that $|KF\setminus \{0\}|<p|F|$. Moreover, the value of the number $l$ used in those arguments can be selected so that $|F|>N$. Therefore another invocation of Lemma 2.25 allows us to conclude that $\mathbb KE$ is exhaustively amenable. This completes the proof of the ``if" assertion. 

We prove the ``only if" statement by establishing its contrapositive. Assume that (i), (ii), (iii), and (iv) are all false. If $E^0=X$, then Theorem~2 implies that $\mathbb KE$ is not amenable and so not exhaustively amenable. Assume that $E^0\neq X$. Henceforth we will employ $P$ to denote the set of all paths whose sources lie outside $X$. Since (ii) and (iii) fail to hold, none of the paths in $P$ contain infinite emitters or cycles, which means that $P$ must be finite. 

According to Proposition~2.23(ii),  it will follow that $\mathbb KE$ is not exhaustively amenable if we can show that there are an integer $N$ and a finite dimensional subspace $U$ of $\mathbb KE$ such that $\dim_{\mathbb K}(UV)\geq 2\dim_{\mathbb K} V$ for all finite dimensional subspaces $V$ with $\dim_{\mathbb K} V>N$. To accomplish this, we make use of the fact that $\mathbb KE_X$ is not amenable, which is a consequence of Theorem~2. By Proposition~2.23(i), this means that there is a finite dimensional subspace $U$ of $\mathbb KE_X$ such that $\dim_{\mathbb K}\left (UV\right )\geq 3\dim_{\mathbb K} V$ for all finite dimensional subspaces $V$ of $\mathbb KE_X$. Write $\widehat{P}:={\rm span}_{\mathbb K} P$. Then $\mathbb KE=\widehat{P}\oplus \mathbb KE_X$ as vector spaces.

Set $N:=3|P|$, and take $V$ to be an arbitrary finite dimensional subspace of $\mathbb KE$ that has dimension larger than $N$. Write $V_{\widehat{P}}$ and $V_{\mathbb KE_X}$ for the images of $V$ under the projection maps $\mathbb KE\to \widehat{P}$ and $\mathbb KE\to \mathbb KE_X$, respectively. Then $V\subseteq V_{\widehat{P}}\oplus V_{\mathbb KE_X}$. Moreover, we have the chain of inequalities
\[\dim_{\mathbb K}\left (UV\right )= \dim_{\mathbb K}\left (UV_{\mathbb KE_X}\right )\geq 3\dim_{\mathbb K} V_{\mathbb KE_X}\geq 3\left (\dim_{\mathbb K} V-\dim_{\mathbb K} V_{\widehat{P}}\right )\geq 3\left (\dim_{\mathbb K} V-|P|\right )>2\dim_{\mathbb K} V .\]
Therefore $\mathbb KE$ is not exhaustively amenable. 
\end{proof}

Below we enunciate versions of Theorems 2 and 3 for right amenability, obtained simply by applying Lemma 2.3(i).

\begin{corollary} Let $\mathbb K$ be a field and $E$ a directed graph with $E^0$ finite. Let $X$ be the smallest subset of $E^0$ such that $M(X)=X$ and $X$ contains all the vertices of all the cycles of $E$. Then the statements (i), (ii), (iii), and (iv) below are equivalent.
\begin{enumerate*}
\item $\mathbb KE$ is right  amenable.
\item $\mathbb KE$ satisfies the left strong rank condition.
\item There is no left $\mathbb KE$-module monomorphism $\left (\mathbb KE \right )^2\to \mathbb KE$.
\item $X\neq E^0$, or $E$ contains an exclusive minimal cycle.
\end{enumerate*}
\end{corollary}

\begin{corollary}  Let $\mathbb K$ be a field and $E$ a directed graph with $E^0$ finite and nonempty. Let $X$ be the smallest subset of $E^0$ such that $M(X)=X$ and that contains all the vertices of all the cycles of $E$. Then $\mathbb KE$ is exhaustively right amenable if and only if at least one of the conditions (i), (ii), (iii), or (iv) below is satisfied.
\begin{enumerate*}
\item $X=\emptyset$. 
\item $M(E^0\setminus X)$ contains an infinite receiver.
\item $M(E^0\setminus X)$ contains a vertex that belongs to a cycle. 
\item $E$ contains an exclusive minimal cycle.
\end{enumerate*}
\end{corollary}

We conclude this section with an example of a graph to which we can easily apply our results. 

\begin{example} {\rm Let $E$ be the graph with $E^0:= \{v_0, v_1, v_2, v_3, v_4, v_5\}$ and edges as shown in the diagram below.
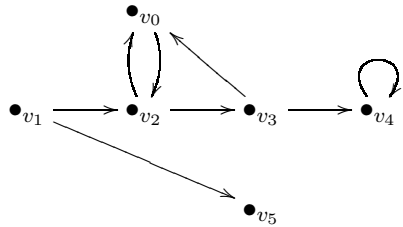
\begin{displaymath}
\xymatrix{
 & \bullet_{v_0} \ar@/^0.5pc/@{->}[d] &  &  & \\
\bullet_{v_1} \ar@{->}[r] \ar@{->}[drr]  & \bullet_{v_2} \ar@{->}[r] \ar@/^0.5pc/@{->}[u] & \bullet_{v_3} \ar@{->}[r] \ar@{->}[ul] & \bullet_{v_4}  \ar@(ul,ur)  \\
 & & \bullet_{v_5} 
}
\end{displaymath}

Let $C:=\{v_0, v_2, v_3, v_4\}$; that is, $C$ is the set of vertices that lie on cycles. Then $T(C)=C$, and $M(C)=C\cup \{v_1\}$. Therefore, for any field $\mathbb K$, the algebra $\mathbb KE$ is exhaustively amenable and exhaustively right amenable by Theorem~3 and Corollary~3.5, respectively. Also, Theorem~1 shows that, for any ring $R$, the ring $RE$ satisfies the strong rank condition if and only if $R$ does. Furthermore, according to Corollary~3.2, the same holds for the left strong rank condition. }
\end{example}

\section{Relative Leavitt path rings} 

In this section, we establish our results about relative Leavitt path rings.  We begin by showing that, when a graph is restricted to a hereditary subset, the relative Leavitt path ring associated to the restricted graph is isomorphic to a nonunital subring of the relative Leavitt path ring arising from the entire graph. This is doubtless well known for Leavitt path algebras over fields, and generalizing it to relative Leavitt path rings with arbitrary coefficients is straightforward.

\begin{lemma} Let $E$ be a directed graph with $E^0$ finite, and let $R$ be a ring. Suppose that $V$ is a subset of ${\rm Reg}(E)$. For any hereditary subset $X$ of $E^0$, there is a multiplicative $R$-bimodule monomorphism $L^{V\cap X}_R(E_X)\to L^V_R(E)$ that restricts to the identity map on $X\cup \left (E_X\right )^1\cup \left (\left (E_X\right )^1\right )^\ast$. 
\end{lemma}

\begin{proof} The defining relations of the two relative Leavitt path rings allow us to define a multiplicative $R$-bimodule homomorphism $L^{V\cap X}_R(E_X)\to L^V_R(E)$ that restricts to the identity map on the set in question. That this map is injective follows from Theorem~2.4. 
\end{proof}

Interesting in its own right, the next proposition describes an instance when the generating number of a restricted relative Leavitt path ring cannot be smaller than the generating number of the entire relative Leavitt path ring. Part (i) of this result generalizes \cite[Lemma~4.3]{Ara}
and is proved in the same manner.  Also, the case of part (ii) where the generating number of the restricted Leavitt path ring is equal to one was proved in \cite[p.~126]{Amenability}. Our argument for the general case is merely an extension of the one in \cite{Amenability}. For the statement of the result and the proof, we require the following concepts and notation.

\begin{definition}{\rm Let $E$ be a directed graph and $V$ a subset of $\text{Reg}(E)$. A set $X\subseteq E^0$ is \emph{$V$-saturated} if, for all $v\in V$, the implication 
\[r\left (s^{-1}(v)\right )\subseteq  X\Longrightarrow v\in X\]
holds.

For any $X\subseteq E^0$, we define the sequence $\left (\Lambda^V_k\left (X\right )\right )_{k=0}^\infty$ of subsets of $E^0$ as follows:
\[\Lambda^V_0\left (X\right ):=X;\ \ \ \Lambda^V_k\left (X\right ):=\Lambda^V_{k-1}\left (X\right )\cup \{v\in V\ :\ r\left (s^{-1}\left (v\right )\right )\subseteq \Lambda^V_{k-1}\left (X\right )\}\ \mbox{for $k\geq 1$}.\]}
\end{definition}

\begin{remark}{\rm The set  $\bigcup_{k=0}^\infty \Lambda^V_k \left (X\right )$ is the smallest $V$-saturated subset of $E^0$ that contains $X$, known as the \emph{$V$-saturated closure} of $X$ (see \cite[Lemma~2.0.7]{LPA}). Moreover, if $X$ is hereditary, then it is plainly the case that  $\bigcup_{k=0}^\infty \Lambda^V_k \left (X\right )$ is hereditary. Hence, for an arbitrary subset $X$ of $E^0$, the set  $\bigcup_{k=0}^\infty \Lambda^V_k \left (T(X)\right )$ is the smallest subset of $E^0$ containing $X$ that is both hereditary and $V$-saturated.
}
\end{remark}

\begin{proposition} Let $E$ be a directed graph with $E^0$ finite, and let $V$ be a subset of ${\rm Reg}(E)$.  Assume that $X$ is a hereditary subset of $E^0$ such that the $V$-saturated closure of $X$ is $E^0$. Then, for any ring $R$, the following two assertions hold.
\begin{enumerate*}
\item  The rings $L^{V\cap X}_R\left (E_X\right )$ and $L^V_R(E)$ are Morita equivalent.
\item The inequality
\[{\rm gn}\left (L^V_R\left (E\right ) \right )\leq {\rm gn}\left (L^{V\cap X}_R\left ( E_X\right ) \right )\]
holds, with equality occurring if one of the generating numbers is infinite. 
\end{enumerate*}
\end{proposition}

\begin{proof} Write $L:=L^V_R(E)$ and $L_X:=L^{V\cap X}_R\left (E_X\right )$. Let $\theta: L_X\to L$ be the map described in Lemma 4.1. Let $\epsilon:=\theta\left (1_{L_X}\right )=\sum_{x\in X} x$. Plainly, we have ${\rm Im}\ \theta\subseteq \epsilon L\epsilon$. We claim that the reverse containment also holds. To show this, we will establish that $\epsilon \alpha\beta^\ast\epsilon\in {\rm Im}\ \theta$ for all $\alpha, \beta\in {\rm Path}(E)$. If $s(\alpha)\notin X$ or $s(\beta)\notin X$, then $\epsilon \alpha\beta^\ast\epsilon=0$. Suppose that $s(\alpha), s(\beta)\in X$. Then $\alpha, \beta\in {\rm Im}\ \theta$, which means  $\epsilon \alpha\beta^\ast\epsilon\in {\rm Im}\ \theta$. Therefore we have ${\rm Im}\ \theta= \epsilon L\epsilon$, so that $L_X\cong \epsilon L\epsilon$. 

Statement (i) will now follow if we can verify that the idempotent $\epsilon$ is full in $L$, that is, that $L\epsilon L=L$. This will follow if we can show that $\Lambda^V_k(X)\subseteq L\epsilon L$ for every $k\geq 0$. We prove this by induction on $k$. The case $k=0$ is obviously true. Assume that $k>0$ and let $v\in \Lambda^V_k(X)$. If $v\in \Lambda^V_{k-1}(X)$, then $v\in L\epsilon L$ by the inductive hypothesis. Assume that $v\notin \Lambda^V_{k-1}(X)$. In this case, $v$ must be regular. Let $e_1,\dots,e_n$ be all the distinct edges that emanate from $v$.  Then $r(e_i)\in \Lambda^V_{k-1}(X)\subseteq L\epsilon L$ for $i=1,\dots, n$. Thus $e_i=e_ir(e_i)\in L\epsilon L$ for $i=1,\dots, n$. Hence $v=\sum_{i=1}^ne_ie_i^\ast\in L\epsilon L$, which completes the argument. 

We now turn to the task of showing statement (ii). The fact that the two generating numbers coincide if one of them is infinite follows immediately from (i) and Proposition~2.18. 
By the same token, if one generating number is finite, then the other is as well. 
Suppose, then, that the two generating numbers are finite and put $n:={\rm gn}\left (L_X\right )$.  For convenience, we will write $i\cdot z:=\underbrace{z\oplus\cdots \oplus z}_i$ for every $z\in L$ and positive integer $i$. Using this notation, we conclude from Proposition 2.12(v) that $m\cdot \epsilon\precsim n\cdot \epsilon$ for every $m>n$. 

We will show that $n\cdot v\precsim n\cdot \epsilon$ for every $v\in E^0$. This will follow if we establish that, for every $k\geq 0$, $n\cdot v\precsim n\cdot \epsilon$ for all $v\in \Lambda^V_k(X)$. We prove this by induction on $k$, the base case being a consequence of parts (ii), (iii), and (v) of Lemma 2.9. Assume that $k>0$ and let $v\in \Lambda^V_k(X)$. If $v\in \Lambda^V_{k-1}(X)$, then $n\cdot v\precsim n\cdot \epsilon$. Suppose that $v\notin \Lambda^V_{k-1}(X)$. Let $e_1,\dots,e_l$ be all the distinct edges that start at $v$. Then $r(e_i)\in \Lambda^V_{k-1}(X)$ for $1\leq i\leq  l$. Hence $n\cdot r(e_i)\precsim n\cdot \epsilon$ for $1\leq i\leq l$. Applying the relevant parts of Lemma 2.9, we deduce the chain of relations
\[n\cdot v=n\cdot \left (\sum_{i=1}^l e_ie_i^\ast\right )\precsim n\cdot \left (\bigoplus_{i=1}^l e_ie_i^\ast\right )\precsim n\cdot \left (\bigoplus_{i=1}^l e_i^\ast e_i\right )\precsim \mathlarger{\mathlarger{\bigoplus}}_{i=1}^l n\cdot r(e_i)\precsim (nl)\cdot \epsilon\precsim n\cdot \epsilon.\] 

Knowing that $n\cdot v\precsim n\cdot \epsilon$ for every $v\in E^0$ allows us to finish the proof. Since the unit element in $L$ is $1=\sum_{v\in E^0} v$, we observe that
\[(2n)\cdot 1\precsim \mathlarger{\mathlarger{\bigoplus}}_{v\in E^0} (2n)\cdot v\precsim (2n|E_0|)\cdot \epsilon\precsim  n\cdot \epsilon\precsim n\cdot \left (\epsilon \oplus \bigoplus_{v\in E_0\setminus X} v\right )\precsim n\cdot 1.\]
It follows, then, from Proposition 2.12(iv) that ${\rm gn}\left (L \right )\leq n$. 
\end{proof}

Next we examine the relative Leavitt path ring of a quotient graph, defined as follows. 

\begin{definition}{\rm Let $E$ be a directed graph and $X$ a subset of $E^0$. The {\it quotient graph} $E/X$ is defined to be the restriction of $E$ to $E^0\setminus X$.}
\end{definition}

If the set $X\subseteq E^0$ is $V$-saturated, then $L_R^{V\setminus X}(E/X)$ is isomorphic to the quotient of $L_R^V(E)$ by the ideal generated by $X$. 

\begin{lemma} Let $R$ be a ring and $E$ a directed graph with finitely many vertices. Suppose that $V$ is a set of regular vertices of $E$, and that $X$ is a $V$-saturated subset of $E^0$. Furthermore, let $I(X)$ be the ideal in $L_R^V(E)$ that is generated by X. Then
\[L^{V\setminus X}_R(E/X)\cong L^V_R(E)/I(X)\]
as $R$-rings. 
\end{lemma}

\begin{proof} The defining relations for relative Leavitt path rings ensure that there is an $R$-ring epimorphism $\phi:L^V_R(E)\to L^{V\setminus X}_R(E/X)$ with the following properties.
\begin{itemize}
\item $\phi(v)=0$ if $v\in X$.
\item$\phi(v)=v$ if $v\in E^0\setminus X$.
\item  If $e\in E^1$ such that $s(e)\in X$ or $r(e)\in X$, then
$\phi(e)=\phi(e^\ast)=0$.
\item If $e\in E^1$ with $s(e)\notin X$ and $r(e)\notin X$, then $\phi(e)=e$ and $\phi(e^\ast)=e^\ast$.
\end{itemize}

Since $I(X)\subseteq {\rm Ker}\ \phi$, the $R$-ring epimorphism $\phi$ induces an $R$-ring epimorphism $\psi:L^V_R(E)/I(X)\to L^{V\setminus X}_R(E/X)$. We will show that $\psi$ is an isomorphism by constructing an inverse. Observe that the defining relations for $L^V_R(E)$ and $L^{V\setminus X}_R(E/X)$ as well as the definition of $I(X)$ yield the existence of an $R$-ring homomorphism $\theta:L^{V\setminus X}_R(E/X)\to L^V_R(E)/I(X)$
such that  the following properties hold: $\theta(v)=v$ for $v\in (E/X)^0$; $\theta(e)=e$ and $\theta(e^\ast)=e^\ast$ for $e\in (E/X)^1$. Then $\psi\theta(y)=y$ for all $y\in L^{V\setminus X}_R(E/X)$, and $\theta\psi(y+I(X))=y+I(X)$ for all $y\in L^V_R(E)$. 
\end{proof}

\begin{remark}{\rm Lemma 4.5 is essentially a special case of \cite[Theorem 2.4.12]{LPA}, although the latter result was stated only for Leavitt path algebras over fields.}
\end{remark}

In the case where $X$ is a hereditary subset of $E^0$, the ideal of $L_R^V(E)$ that is generated by $X$ admits the following description. 

\begin{lemma}[{\cite[Lemma 2.4.1]{LPA}}] Let $E$ be a directed graph with $E^0$ finite. Let $V$ and $X$ be subsets of $E^0$ such that $V$ contains only regular vertices and $X$ is hereditary. Then the ideal of $L_R^V(E)$ that is generated by $X$ is generated as an $R$-module by the set
\[\{\alpha\beta^\ast\ :\ \alpha, \beta\in {\rm Path}(E)\ \mbox{\rm and}\ r(\alpha)=r(\beta)\in X\}.\]
\end{lemma}

\begin{proof} Let $M$ be the $R$-module generated by the set in question, and let $I$ be the ideal of $L_R^V(E)$ that is generated by $X$. It follows immediately from the relations (2.1) that $M$ is an ideal in $L^V_R(E)$. Moreover, it is plain that $X\subseteq M$, and that $M\subseteq I$. Therefore $M=I$.
\end{proof}

The quotient described in Lemma 4.5  is a direct summand when $X$ is a hereditary set that contains all the vertices of cycles and $E^0\setminus X$ fails to have any infinite emitters. This was proved for coefficient rings that are fields in \cite{Amenability}, and we use the ideas from there to construct a proof for arbitrary coefficient rings. 

\begin{lemma}[{Ara, Li, Lled\'o, Wu \cite[p. 294]{Amenability}}] Let $R$ be a ring and $E$ be a directed graph with $E^0$ finite. Let $V$ be a subset of ${\rm Reg}(E)$, and take $X$ to be a hereditary and $V$-saturated subset of $E^0$ such that $X$ contains all the vertices that belong to cycles and $E^0\setminus X$ has no infinite emitters. Let $I(X)$ be the ideal in $L_R^V(E)$ that is generated by $X$. Then
$I(X)$ has a unit element, and  
\[L^V_R(E)\cong L^{V\setminus X}_R(E/X)\oplus I(X)\]
as $R$-rings.
\end{lemma}

\begin{proof} For every $v\in E^0$, define $P(v,X)$ to be the set of all the paths from $v$ to $X$ whose nonterminal vertices all lie outside of $X$. 
Because $E^0\setminus X$ has no infinite emitters, we know that each of these sets must be finite. 
Also, for every $x\in X$, we have $P(x,X):=\{x\}$. Notice too  that, for any $v\in E^0$ and $\mu,\lambda\in P(v,X)$, we have $\mu^\ast \lambda\neq 0$ if and only if $\mu=\lambda$. 
This observation follows immediately from the relations (2.1). 

Now write
\[u:=\mathlarger{\sum}_{v\in E^0}\sum_{\mu\in P(v,X)} \mu \mu^\ast.\]
It is easy to see that $u\in I(X)$. We claim further that $u$ is central in $L^V_R(E)$. Since $\mu\mu^\ast$ commutes with every vertex, the same is true for $u$.  
To show that $u$ also commutes with every edge, let $e$ be an edge with $s(e)=v$ and $r(e)=w$. If $v=w$, then $v\in X$ and so $ue=e=eu$. Suppose that $v\neq w$. Then
\[ue=\sum_{\mu\in P(v,X)}\mu\mu^\ast e = \sum_{\nu\in P(w,X)} e\nu\left (e\nu\right )^\ast e = \sum_{\nu\in P(w,X)} e\nu \nu^\ast = eu.\]
It follows that $u$ also commutes with $e^\ast$. 
Therefore $u$ is indeed central in $L^V_R(E)$.

We also maintain that $u$ is a unit element for $I(X)$. To verify this, let $\alpha, \beta$ be paths such that $r(\alpha)=r(\beta)\in X$. Then $\alpha=\alpha^\prime \alpha^{\prime\prime}$, where $\alpha'\in P(s(\alpha),X)$. Writing $v:=s(\alpha)$, we have
\[u\alpha\beta^\ast=\sum_{\mu\in P(v,X)} \mu\mu^\ast\alpha\beta^\ast = \sum_{\mu\in P(v, X)} \mu\mu^\ast \alpha^\prime \alpha^{\prime\prime}\beta^\ast.\]
Furthermore, for every $\mu\in P(v,X)$, $\mu^\ast \alpha^\prime\neq 0$ if and only if $\mu=\alpha^\prime$, in which case $\mu^\ast\alpha^\prime= r(\alpha^\prime)$. Thus $u \alpha\beta^\ast=\alpha^\prime \alpha^{\prime\prime}\beta^\ast=\alpha\beta^\ast$. Invoking Lemma 4.6, we conclude that $u$ is a unit element  for $I(X)$.  We therefore have $L^V_R(E)\cong I(X)\oplus (1-u)L^V_R(E)$ as $R$-rings. An appeal to Lemma 4.5, then, furnishes the desired conclusion. 
\end{proof}

Now we prove our main results about relative Leavitt path rings, which have Theorems 4 and 5 from the introduction as special cases. 

\begin{theorem}  Let $R$ be a ring and E a directed graph with finitely many vertices. Let $V$ be a subset of ${\rm Reg}(E)$. Define $X$ to be the smallest hereditary and $V$-saturated subset of $E^0$ that contains all the vertices of all the cycles of $E$. Then the statements (i), (ii), and (iii) below are equivalent. 
\begin{enumerate*}
\item $L^V_R(E)$ satisfies the strong rank condition.
\item $R$ satisfies the strong rank condition, and there is no $L^V_R(E)$-module monomorphism $L^V_R(E)\oplus L^V_R(E)\to L^V_R(E)$. 
\item $R$ satisfies the strong rank condition, and at least one of the statements (1) or (2) holds:
\begin{enumeratenum}
\item $X\neq E^0$;
\item $E$ contains an exclusive maximal cycle.
\end{enumeratenum}
\end{enumerate*}
\end{theorem}

\begin{theorem}  Let $R$ be a ring and E a directed graph with finitely many vertices. Let $V$ be a subset of ${\rm Reg}(E)$. Define $X$ to be the smallest hereditary and $V$-saturated subset of $E^0$ that contains all the vertices of all the cycles of $E$. Then the statements (i), (ii), and (iii) below are equivalent. 
\begin{enumerate*}
\item $L^V_R(E)$ satisfies the rank condition.
\item $R$ satisfies the rank condition, and ${\rm gn}\left (L^V_R\left (E\right )\right )>1$. 
\item $R$ satisfies the rank condition, and at least one of the statements (1) or (2) holds:
\begin{enumeratenum}
\item $X\neq E^0$;
\item $E$ contains an exclusive maximal cycle.
\end{enumeratenum}
\end{enumerate*}
\end{theorem}

We will present a single proof that will establish both theorems. As related in the introduction, our argument tracks the one for  Ara, Li, Lled\'o, and Wu's result \cite[Theorem 5.10]{Amenability} characterizing amenable Leavitt path algebras very closely. The chief difference is our use of Proposition 2.20 in conjunction with their reasoning. 

\begin{proof}[Proof of Theorems 4.8 and 4.9] Throughout the proof, we will write $L:=L^V_R(E)$, $X^\prime:=E^0\setminus X$, $L_{X^\prime}:=L^{V\setminus X}_R(E/X)$, and $\Sigma$ for a particular $R$-basis for $L$ of the type described in Proposition 2.2. 
Let $\mathcal{C}$ represent either the rank condition or the strong rank condition. In the case that $\mathcal{C}$ stands for the rank condition, then $\mathcal{C}_1$ will represent the property that the generating number is larger than one. On the other hand, if $\mathcal{C}$ is interpreted as the strong rank condition, then $\mathcal{C}_1$ will be taken to be the property that there is no 
$L$-module monomorphism $L^2\to L$. With this notation, the statements whose equivalence is to be established can be expressed as follows.
\begin{enumerate*}
\item $L$ satisfies $\mathcal{C}$.
\item $R$ satisfies $\mathcal{C}$, and $L$ satisfies $\mathcal{C}_1$. 
\item $R$ satisfies $\mathcal{C}$, and at least one of the statements (1) or (2) holds:
\begin{enumeratenum}
\item $X\neq E^0$;
\item $E$ contains an exclusive maximal cycle.
\end{enumeratenum}
\end{enumerate*}

The implication (i)$\implies$(ii) follows from the definitions of the rank condition and strong rank condition, as well as Lemma 2.7 and Corollary 2.17. Our plan is to first prove (iii)$\implies$(i) and then (ii)$\implies$(iii). 
To establish the former implication, we assume that $R$ satisfies $\mathcal{C}$.  First we show that, if $X^\prime$ has an infinite emitter, then $L$ satisfies $\mathcal{C}$. Suppose that $X^\prime$ has an infinite emitter, and let $y$ be an infinite emitter in $X^\prime$ that is maximal with respect to the quasiordering restricted to $X'$. 
Take $K$ to be a finite subset of $\Sigma$. Then there is an edge $e\in s^{-1}(y)$ such that $e$ is not contained in any path $\beta$ for which $\alpha \beta^\ast\in K$ for some path $\alpha$. Let 
\[F:=\{\mu\, e\ :\ \mu\in {\rm Path}(E)\ \mbox{and}\ r(\mu)=y\}.\]
Observe that none of the paths in $F$ contain a cycle or an infinite emitter unequal to $y$. As a consequence, the set $F$ must be finite. 

We will show that $KF\setminus \{0\}\subseteq F$, which will imply, by Proposition~2.20, that $L$ satisfies $\mathcal{C}$. To accomplish this, let $\alpha, \beta\in {\rm Path}(E)$ such that $\alpha \beta^\ast\in K$, and let $\mu$ be a path with terminus $y$. If $\alpha \beta^\ast\mu e\neq 0$, then there must be a path $\lambda$ such that $\mu=\beta\lambda$. In this case, we have $\alpha\beta^\ast\mu e=\alpha\lambda e\in F$, proving our claim. 

Henceforth we will assume that $X^\prime$ fails to contain any infinite emitters. Suppose now that statement (1) is true. 
By Lemma 4.7,  we have $L\cong L_{X^\prime}\oplus I(X)$. Also, Lemma~2.13 implies that $L_{X^\prime}$ satisfies $\mathcal{C}$. It follows, then, from Lemma~2.8 that $L$ satisfies~$\mathcal{C}$. 

To complete the proof of (iii)$\implies$(i), we suppose that (2) holds. Arguing exactly as in the proof of \cite[Theorem 5.10]{Amenability}, we let $\omega$ be a maximal cycle that is exclusive, and let $v$ be the basepoint of $\omega$. The subgraph $E_{M(v)}$, therefore, contains $\omega$ and fails to contain any cycle with an edge set different from that of $\omega$. 
We maintain further that there are no infinite emitters in $E_{M(v)}$. By assumption, $M(v)\setminus X$ has no infinite emitters in $E$; hence we must only show that $X$ fails to contain any infinite emitters for $E_{M(v)}$. Notice that the presence of an edge in $E_{M(v)}^1\setminus {\rm Edge}(\omega)$ emanating from a vertex on $\omega$ would give rise to a cycle in $E_{M(v)}$ with a set of edges different from ${\rm Edge}(\omega)$. This means that each vertex in $M(v)\cap {\rm Vert}(\omega)$ emits merely a single edge in $E_{M(v)}$. 
Let $C$ be the set of all the vertices in $E^0$ that are on cycles. We argue that $T(C)\cap M(v)={\rm Vert}(\omega)$. To show this, suppose that $x\in T(C)\cap M(v)$. Then there is a cycle $\xi$ and a path $\delta$ that starts on $\xi$ and ends at $x$. But the maximality of $\omega$, then, yields a path from $v$ to $\xi$. These circumstances place $x$ on a cycle, so that $x\in {\rm Vert}(\omega)$. Hence it is indeed true that $T(C)\cap M(v)={\rm Vert}(\omega)$. Furthermore, $\Lambda^V_k\left (T\left (C\right )\right )\backslash T(C)$ contains, by definition, only regular vertices for every nonnegative integer $k$.  Since $X=\bigcup_{k=0}^\infty \Lambda^V_k\left(T\left (C\right )\right )$, this completes the proof that $X$ does not contain any infinite emitters for $E_{M(v)}$. 

The absence of any infinite emitters in $E_{M(v)}$ and any cycles with edge sets different from ${\rm Edge}(\omega)$ guarantees that there are only finitely many paths in the subgraph that do not contain $\omega$. Making use of this fact, we will apply Proposition 2.20 to show that $L$ fulfills $\mathcal{C}$. For this purpose, we let $K:=\{\alpha_i\beta^\ast_i:i=1,\dots, n\}$ be a finite subset of $\Sigma$ and let $p$ be a real number greater than $1$.  Let $N_1$ be a positive integer such that $N_1|\omega|\geq  {\rm max}\{|\alpha_i|,|\beta_i|\}$ for all $i=1,\dots, n$. Next pick an integer $N_2>N_1$ such that $\frac{N_1+N_2+1}{N_2-N_1} < p$. 

Denote by $P$ the set of all paths that end at $v$ and that do not contain $\omega$. Then $P$ is finite and nonempty. Define $F$ to be the set of all paths of the form $\gamma \omega^k$ such that $\gamma\in P$ and $N_1<k\leq N_2$.  We will show that 
\begin{equation} KF\setminus \{0\}\subseteq \{\gamma \omega^k\ :\ 1\leq k\leq N_1+N_2+1\ \ \mbox{and}\ \ \gamma\in P\}.\end{equation}
To establish this containment, let $i\in \mathbb Z\cap [1,n], \gamma\in P$, and $k\in \mathbb Z\cap (N_1,N_2]$ such that $\left (\alpha_i\beta_i^\ast\right )\left (\gamma\omega^k\right )\neq 0$. Since $k|\omega|>|\beta_i|$, the relations (2.1) imply that 
\[\left ( \alpha_i\beta_i^\ast\right )\left (\gamma\omega^k\right )=\alpha_i\eta\omega^{k'}\]
for some $\eta\in P$ and integer $k'\in [1, N_2]$. If $\alpha_i$ fails to contain $\omega$, then it is clearly the case that $\left (\alpha_i\beta_i^\ast\right )\left (\gamma\omega^k\right )$ lies in the set on the right of the containment (4.1). Suppose that $\alpha_i$ contains $\omega$. Then there are a positive integer $n\leq N_1$ and paths $\lambda, \mu$ such that $\alpha_i=\lambda\omega^n\mu$ and $\lambda, \mu$ both fail to contain $\omega$. Hence
\[\left (\alpha_i\beta_i^\ast\right )\left (\gamma\omega^k\right )=\lambda\omega^n\mu\eta\omega^{k'}.\]
Also, we have $s(\mu)=r(\eta)=v$. Thus, since $\omega$ is exclusive, it follows that either $\mu=\eta=v$ or $\mu\eta=\omega$. As a consequence, we have that either $\left (\alpha_i\beta_i^\ast\right )\left (\gamma\omega^k\right )=\lambda\omega^{n+k'}$ or $\left (\alpha_i\beta_i^\ast\right )\left (\gamma\omega^k\right )=\lambda\omega^{n+k'+1}$.
Therefore (4.1) holds. 

The containment (4.1) implies that $|KF\setminus \{0\}|\leq |P|(N_1+N_2+1)$. Consequently, since $|F|=|P|(N_2-N_1)$, we conclude that
$|KF\setminus \{0\} |< p|F|$. Therefore, by Proposition~2.20, the ring $L$ satisfies $\mathcal{C}$. This completes the proof of (iii)$\implies$(i). 

It remains to establish the implication (ii)$\implies$(iii). We will prove the contrapositive. Thus we assume that  $X=E^0\neq \emptyset$ and every maximal cycle of $E$ is nonexclusive. Our aim is to show that
${\rm gn}(L)=1$, which will imply that $L$ fails to satisfy $\mathcal{C}_1$. 
According to Proposition 4.3(ii), it will follow that ${\rm gn}(L)=1$ if we can show that ${\rm gn}\left (L_R^{V \cap T(C)}\left (E_{T(C)}\right )\right ) = 1$. Hence there is no real loss of generality in assuming that $E=T(C)$.  

Since every maximal cycle in $E$ is nonexclusive, we can find distinct maximal cycles $\omega_1,\dots,\omega_n,\xi_1,\dots,\xi_n$ with the following properties.
\begin{itemize}
\item For $i=1,\dots, n$, the cycles $\omega_i$ and $\xi_i$ are based at the same vertex $v_i$.
\item $\displaystyle{v_i=v_j \Longleftrightarrow i=j}$.
\item For every cycle $\theta$, there is an integer $i\in [1,n]$ such that $\omega_i\geq \theta$.
\end{itemize}

Notice that $\displaystyle{v_i\oplus v_i\precsim v_i}$ for $i=1,\dots, n$. This is a consequence of the matrix equation

\[\begin{pmatrix}  \omega_i^\ast\\ 
  \xi_i^\ast \end{pmatrix} \begin{pmatrix}  \omega_i &  \xi_i \end{pmatrix}=\begin{pmatrix} v_i & 0\\
0 & v_i\end{pmatrix}.\]
Set $x:=\sum_{i=1}^n v_i$. Invoking Lemma 2.9(i)(iii)(iv)(v), we surmise that $\displaystyle{x\oplus x\precsim x}$. We claim next that $v\precsim x$ for every vertex $v$. To verify this, take $v$ to be an arbitrary vertex. We can find a path $\gamma$ from $v_k$ to $v$ for some integer $k\in [1,n]$. We then have $\gamma^\ast v_k\gamma=v$, which means  $v\precsim v_k$. It follows from Lemma 2.9(ii)(v) that $v\precsim x$. These same two parts of Lemma 2.9, together with parts (iii) and (iv) of that lemma, then allow us to obtain the chain of relations
\[1\oplus 1\precsim \bigoplus_{v\in E^0} v\oplus \bigoplus_{v\in E^0} v\precsim \underbrace{x\oplus\cdots \oplus x}_{2|E^0|}\precsim x\precsim 1.\]
Thus ${\rm gn}(L)=1$ by Proposition 2.12(iv). 
\end{proof}

Below we state the special cases of Theorems 4.8 and 4.9 for Cohn path rings; the second one generalizes \cite[Corollary 3.18]{UGN}. 

\begin{corollary}  Let $R$ be a ring and E a directed graph with finitely many vertices.  Let $X$ be the smallest hereditary subset of $E^0$ containing all the vertices of all the cycles of $E$. Then the statements (i), (ii), and (iii) below are equivalent. 
\begin{enumerate*}
\item $C_R(E)$ satisfies the strong rank condition.
\item $R$ satisfies the strong rank condition, and there is no $C_R(E)$-module monomorphism $C_R(E)\oplus C_R(E)\to C_R(E)$. 
\item $R$ satisfies the strong rank condition, and at least one of the statements (1) or (2) holds:
\begin{enumeratenum}
\item $X\neq E^0$;
\item $E$ contains an exclusive maximal cycle.
\end{enumeratenum}
\end{enumerate*}
\end{corollary}

\begin{corollary}  Let $R$ be a ring and E a directed graph with finitely many vertices.  Let $X$ be the smallest hereditary subset of $E^0$ that contains all the vertices of all the cycles of $E$. Then the statements (i), (ii), and (iii) below are equivalent. 
\begin{enumerate*}
\item $C_R(E)$ satisfies the rank condition.
\item $R$ satisfies the rank condition, and ${\rm gn}\left (C_R\left (E\right )\right )>1$. 
\item $R$ satisfies the rank condition, and at least one of the statements (1) or (2) holds:
\begin{enumeratenum}
\item $X\neq E^0$;
\item $E$ contains an exclusive maximal cycle.
\end{enumeratenum}
\end{enumerate*}
\end{corollary}

\begin{remark}{\rm In view of Lemma 2.3(ii), the statement of Theorem 4.8 remains true if the three occurrences  of ``the strong rank condition" are replaced by ``the left strong rank condition" 
and the phrase ``no $L_R^V(E)$-module monomorphism" in (ii) is replaced by ``no left $L_R^V(E)$-module monomorphism." A similar observation can be made about Corollary 4.10.}
\end{remark}

Next we state the results from \cite[Theorem 5.10]{Amenability} about amenability and exhaustive amenability for Leavitt path algebras. We express them in a form that applies to relative Leavitt path algebras and, for the first result, that also includes our results about the rank condition and strong rank condition.  

\begin{theorem} Let $\mathbb K$ be a field and $E$ a directed graph with $E^0$ finite. Let $V$ be a subset of ${\rm Reg}(E)$. Define $X$ to be the smallest hereditary and $V$-saturated subset of $E^0$ that contains all the vertices of all the cycles of $E$. Then the statements (i), (ii), (iii), (iv), (v), and (vi) below are equivalent, where $L:=L_{\mathbb K}^V(E)$.
\begin{enumerate*}
\item $L$ is amenable.
\item  $L$ satisfies the strong rank condition.
\item There is no $L$-module monomorphism $L^2\to L$. 
\item $L$ satisfies the rank condition.
\item ${\rm gn}(L)>1$. 
\item $X\neq E^0$, or $E$ contains an exclusive maximal cycle.
\end{enumerate*}
\end{theorem}

The part of Theorem 4.12 that is drawn from \cite[Theorem 5.10]{Amenability} is the equivalence of (i), (v), and (vi) for the case $V={\rm Reg}(E)$. 

\begin{theorem}[{Ara, Li, Lled\'o, Wu \cite[Theorem 5.10]{Amenability}}]  Let $\mathbb K$ be a field and $E$ a directed graph with $E^0$ finite and nonempty. Let $V$ be a subset of ${\rm Reg}(E)$. Define $X$ to be the smallest hereditary and $V$-saturated subset of $E^0$ that contains all the vertices of all the cycles of $E$. Then the statements (i) and (ii) below are equivalent.
\begin{enumerate*}
\item $L^V_{\mathbb K}(E)$ is exhaustively amenable.
\item  At least one of the following three statements holds:
\begin{enumeratenum}
\item $X=\emptyset$;
\item $E^0\setminus X$ contains an infinite emitter;
\item $E$ contains an exclusive maximal cycle.
\end{enumeratenum}
\end{enumerate*}
\end{theorem}

\begin{remark}{\rm Note that, for any field $\mathbb K$ and directed graph $E$, we have, by Lemma 2.3(ii), that $\left (L_{\mathbb K}^V(E)\right )^{\rm op}\cong L^V_{\mathbb K}(E)$ for any set $V\subseteq \text{Reg}(E)$. As a result, a relative Leavitt path algebra over a field is right amenable (respectively, exhaustively right amenable) if and only if it is amenable (respectively, exhaustively amenable). Similarly, it satisfies the left strong rank condition if and only if it satisfies the strong rank condition.}
\end{remark}

We conclude the paper with an example illustrating the utility of our results about relative Leavitt path rings.

\begin{example}{\rm Let $E$ be the graph with $E^0:=\{v_0, v_1, v_2, v_3, v_4, v_5, v_6\}$ and edges as described by the following diagram. 

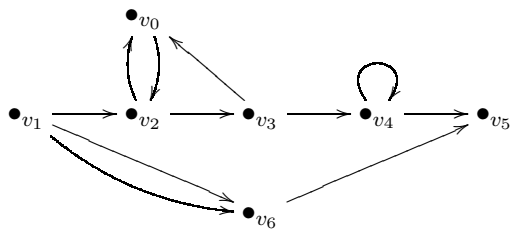
\begin{displaymath}
\xymatrix{
 & \bullet_{v_0} \ar@/^0.5pc/@{->}[d] &  &  & \\
\bullet_{v_1} \ar@{->}[r] \ar@{->}[drr] \ar@/_0.8pc/@{->}[drr]& \bullet_{v_2} \ar@{->}[r] \ar@/^0.5pc/@{->}[u] & \bullet_{v_3} \ar@{->}[r] \ar@{->}[ul] & \bullet_{v_4} \ar@{->}[r] \ar@(ul,ur) & \bullet_{v_5} \\
 & & \bullet_{v_6} \ar@{->}[urr] & &
}
\end{displaymath}

The set of vertices of $E$ that are located on cycles is $C:=\{v_0, v_2, v_3, v_4\}$, and $T(C)=\{v_0, v_2, v_3, v_4, v_5\}$.
Observe also that the saturated closure of $T(C)$ is $E^0$. Therefore, according to Theorem 4.9, $L_R(E)$ does not satisfy the rank condition for any ring $R$. However, by Corollary 4.11, $C_R(E)$ satisfies the rank condition if and only if $R$ does, and the same holds for the strong rank condition by Corollary 4.10. Moreover, Theorem 4.12 implies that, for every field $\mathbb K$, the algebra $L_{\mathbb K}(E)$ is nonamenable but $C_{\mathbb K}(E)$ is amenable. The latter algebra, however, fails to be  exhaustively amenable by Theorem 4.13.

We also point out that the results of Section~3 show that $\mathbb KE$ is both exhaustively amenable and exhaustively right amenable. Furthermore, $RE$ satisfies the left or right strong rank condition if and only if $R$ satisfies the respective condition. 
}
\end{example}

\vspace{20pt}

\noindent {\sc Karl Lorensen}\\
Department of Mathematics and Statistics\\
Pennsylvania State University, Altoona College\\
Altoona, PA 16601, USA\\
E-mail: {\tt kql3@psu.edu}
\vspace{20pt}

\noindent {\sc Johan \"Oinert}\\
Department of Mathematics and Natural Sciences\\
Blekinge Institute of Technology\\
SE-37179 Karlskrona, Sweden\\
E-mail: {\tt johan.oinert@bth.se}
\vspace{3pt}

\noindent {\it and}
\vspace{3pt}

\noindent Department of Engineering\\
University of Sk\"{o}vde\\
SE-54128 Sk\"{o}vde, Sweden

\end{document}